%% file: pof-decomp.tex
\documentclass[a4paper, reqno]{amsart}
\usepackage[utf8]{inputenc}
\usepackage[ngerman, english]{babel}
\usepackage{graphicx}
\usepackage{svg}
\usepackage{url}
\usepackage{color,soul}
\definecolor{Azure}{rgb}{0.2,0.8,1}
\setulcolor{Azure}

\graphicspath{{images/}}
\usepackage{amsmath, amssymb}
\usepackage{amsthm, thmtools}

\usepackage{enumerate}
\usepackage{nicematrix}
\usepackage{marvosym}

\usepackage{hyperref} 
\usepackage{cleveref}

\usepackage{csvsimple}
\usepackage{tabularx}
\usepackage{algorithm}
\usepackage{algorithmic}

\input{general-definitions.tex}

\input{specific-definitions.tex}

\begin{document}


\begin{abstract} 
	We consider simultaneous Waring decompositions: Given forms $ f_d $ of degrees $ kd $, $ (d = 2,3 )$, which admit a representation as $ d $-th power sums of $ k $-forms $ q_1,\ldots,q_m $, when is it possible to reconstruct the addends $ q_1,\ldots,q_m $ from the power sums $ f_d $? Such powers-of-forms decompositions model the moment problem for mixtures of centered Gaussians. The novel approach of this paper is to use semidefinite programming in order to perform a reduction to tensor decomposition. 
	The proposed method works on typical parameter sets at least as long as $ m\leq n-1 $, where $ m $ is the rank of the decomposition and $ n $ is the number of variables. While provably not tight, this  analysis still gives the currently best known rank threshold for decomposing third order powers-of-forms, improving on previous work  \cite{Ge_Huang_Kakade_2015}, which required $ \Omega(1) \leq m \leq \mathcal{O}(\sqrt{n}) $ and, more recently, Bafna, Hsieh, Kothari and Xu \cite{Bafna_Hsieh_Kothari_Xu_2022}, which can go up to $ m = \mathcal{O}(\frac{n}{\log(n)^2}) $. Our algorithm can produce proofs of uniqueness for specific decompositions. 
	A numerical study is conducted on Gaussian random trace-free quadratics, giving evidence that the success probability converges to $ 1 $ in an average case setting, as long as $ m = n $ and $ n\to \infty $. Some evidence is given that the algorithm also succeeds on instances of rank $ m = \Theta(n^2) $. 
\end{abstract}

\title[Semidefinite Powers-of-Forms decomposition]{Unique Powers-of-forms decompositions from simple Gram spectrahedra}
\date{\today}
\author[Blomenhofer]{Alexander Taveira Blomenhofer}
\thanks{Centrum Wiskunde \& Informatica, Amsterdam}
\address{CWI, Networks \& Optimization, Amsterdam, Science Park 123, NL-1098 XG.}
\email{atb@cwi.nl}
\urladdr{cwi.nl/en/people/filipe-alexander-taveira-blomenhofer}
\maketitle

\section{Introduction}\label{sec:introduction}
\input{intro.tex}

\section{Preliminaries}\label{sec:prelims}
\input{prelims.tex}

\section{Basic Algorithm for PoF decomposition}\label{sec:pof-algo}

\input{pof-algo-basic.tex}

\pagebreak
\section{Interpretation of Requirements}\label{sec:interpretation}
\input{interpretation.tex}

\section{Applications}\label{sec:applications}
\input{applications.tex}

\bibliography{bibML}
\bibliographystyle{acm}
\appendix
\input{appendix.tex}

\end{document}

%% file: general-definitions.tex
\theoremstyle{standard}
\declaretheorem[name=Theorem]{theorem}
\numberwithin{theorem}{section}
\numberwithin{equation}{section}
\declaretheorem[sibling=theorem, name=Lemma]{lemma}
\declaretheorem[sibling=theorem, name=Proposition]{prop}

\declaretheorem[sibling=theorem, name=Definition]{defi}
\declaretheorem[sibling=theorem, name=Corollary]{corollary}

\declaretheorem[sibling=theorem, name=Remark]{remark}

\declaretheorem[sibling=theorem, name=Question]{question}

\declaretheorem[sibling=theorem, name=Reminder]{reminder}

\declaretheorem[sibling=theorem, name=Conjecture]{conjecture}

\newcommand{\Q}{\mathbb{Q}}
\newcommand{\R}{\mathbb{R}}
\newcommand{\C}{\mathbb{C}}
\newcommand{\N}{\mathbb{N}}

\renewcommand\P{\mathbb P}

\newcommand{\E}{\mathbb E}

\newcommand{\dual}{\vee}

\DeclareMathOperator{\im}{im}

\DeclareMathOperator{\rk}{rank}
\DeclareMathOperator{\tr}{tr}

\DeclareMathOperator{\conv}{conv}
\DeclareMathOperator{\gram}{Gram}

\DeclareMathOperator{\minimize}{minimize}

\DeclareMathOperator{\find}{find}
\DeclareMathOperator{\given}{given}
\DeclareMathOperator{\subjto}{s. t.}



%% file: specific-definitions.tex
\DeclareMathOperator{\pof}{PoF}
\DeclareMathOperator{\rel}{rel}
\DeclareMathOperator{\relint}{relint}
\DeclareMathOperator{\sosupp}{sosupp}

\DeclareMathOperator{\suppf}{suppf}
\DeclareMathOperator{\length}{len}
\DeclareMathOperator{\psd}{PSD}

\newcommand{\irel}{I_{\rel}}

\usepackage{pifont}

%% file: intro.tex
Waring decompositions for polynomials are a highly studied problem with a wide range of applications in sciences and statistics, including phylogenetics \cite{Landsberg_2012}, cryogenic electron microscopy \cite{Bandeira_Blum-Smith_Kileel_Perry_Weed_Wein_2017}, Gaussian mixtures (\cite{Ge_Huang_Kakade_2015},\cite{Garg_Kayal_Saha_2020}, see \Cref{sec:applications}) and many more \cite{Landsberg_2012}. They serve as a fundamental model in the theory of arithmetic circuits \cite{Garg_Kayal_Saha_2020} and occur as an important algorithmic primitive for various machine learning problems \cite{Barak_Kelner_Steurer_2015},\cite{Ma_Shi_Steurer_2016}. Formally, for a fixed form, i.e., a homogeneous polynomial $ f_d \in K[X]_{dk}$ of degree $ dk \in \N$, the representation
\begin{align}\label{eq:waringdecomp}
	f_d = \sum_{i = 1}^{m} \lambda_i q_i^d, \qquad (q_1,\ldots,q_m\in K[X]_{k},\: \lambda\in K^m)
\end{align}
is a $ k $-\emph{Waring decomposition} of \emph{rank} $ m $ of $ f_d $ over the field $ K $.  For $ f_d\in K[X] $, the minimum $ m $ such that $ f_d $ has a $ k $-Waring decomposition of rank $ m $ is called the $ k $-\emph{Waring rank} of $ f_d $ over the field $ K $. Classically, the main focus of attention used to be the case where $ K = \C $ is the field of complex numbers, and power sums of \emph{linear} polynomials were considered. The latter corresponds to $ k = 1 $. Note that over the complex field, the weights $ \lambda_i $ are redundant, and thus omitted. A long series of work, started more than a century ago, e.g., by Sylvester \cite{Sylvester_1904}, Hilbert \cite{Hilbert_1933_letter} and Terracini \cite{Te12}, lead via the celebrated Alexander-Hirschowitz theorem \cite{hirschowitz1995polynomial} and results by  Chiantini-Ottaviani-Vannieuwenhoven  \cite{Chiantini_Ottaviani_2012},\cite{Chiantini_Ottaviani_Vannieuwenhoven_2014},\cite{Chiantini_Ottaviani_Vannieuwenhoven_2016}, Galuppi-Mella \cite{galuppi2019identifiability} (and many others) to a complete classification when \emph{generic} $ 1 $-Waring decompositions are unique for the forms they describe. This property is called \emph{generic identifiability}. 

A second line of thoughts (e.g., \cite{Harshman_1970},\cite{Leurgans_Ross_Abel_1993},\cite{Anandkumar_Ge_Hsu_Kakade_Telgarsky_2012},\cite{Lathauwer_Castaing_Cardoso_2007}) developed from the side of applications and was continued by theoretical computer scientists. Among many other things, it produced \emph{algorithmic uniqueness theorems} for $ 1 $-Waring decompositions. From one perspective, these can be seen as ``efficient'' algorithms that recover the representation $ f_d = \ell_{1}^d + \ldots + \ell_{m}^d $ from the $ d $-th power sum $ f_d $ as input, under some restrictive assumptions on the rank and some nondegeneracy assumptions on the parameters of the representation, cf. \Cref{thm:undercomplete-waringdecomp} and \cite[Theorem 2.4.8]{Taveira_Blomenhofer_Thesis}. From another perspective, these results provide a proof of uniqueness for the minimum rank decomposition, whenever certain explicit conditions are met. Therefore, they are a tool to produce rank lower bounds for explicit families of polynomials.  

The work on higher degree Waring decompositions has been pioneered by Reznick \cite{Reznick_1992},\cite{Reznick_2003},\cite{Reznick_Tokcan_2017} (with Tokcan), \cite{Reznick_2010},\cite{Reznick_2013},\cite{Reznick_2013_2},\cite{Reznick_2019},\cite{Reznick_2021}.
In recent years, geometers have been trying to understand {uniqueness} and {generic rank} of Waring decompositions also for higher values of $ k $ \cite{Froeberg_Ottaviani_Shapiro_2012}. A conjecture due to Ottaviani \cite[Conjecture 1.2]{Lundqvist_Oneto_Reznick_Shapiro_2019} states that the generic rank of $ k $-Waring decompositions behaves as expected from counting parameters, if $ d\ge 3 $. In \cite{Blomenhofer_Casarotti_Michalek_Oneto_2022}, the author showed in joint work with Casarotti, Michalek and Oneto, that for ``most'' subgeneric ranks $ m $, Waring decompositions are unique, based on work of Nenashev \cite{nenashev2017note} and Casarotti-Mella \cite{Casarotti_Mella_2022}. The results imply in particular bounds on the generic rank for $ k $-Waring decompositions. Casarotti and Postinghel \cite{Casarotti_Postinghel_2023} then studied a different asymptotic setting where not the number of variables $ n $ but rather the degree $ k $ is assumed to be large. Simultaneous Waring decompositions for vectors of forms of (possibly) different degrees have been studied geometrically e.g. by Angelini, Galuppi, Mella and Ottaviani \cite{Angelini_Galuppi_Mella_Ottaviani_2018}.
From the computational perspective, the work of Garg, Kayal and Saha \cite{Garg_Kayal_Saha_2020} and Bafna, Hsieh, Kothari, Xu \cite{Bafna_Hsieh_Kothari_Xu_2022} examined polynomial-time recovery procedures for some variants of powers-of-forms decomposition.  

\smallskip
This paper aims to generalize the second line of work, concerning algorithmic uniqueness theorems, to higher values of $ k $, although a slightly different setting is considered: 
The focus is on \emph{real} decompositions, degree $ k\ge 2 $, and simultaneous power sum decomposition in various degrees. Formally, we provide an algorithm and a uniqueness theorem (cf. \Cref{algo:pof-recovery-general} and \Cref{thm:main-result}) for third order powers-of-forms (POF) decompositions, which have the following basic template:
\begin{align}\label{eq:pof-statement-into}
	(\pof)_{f,m,k} \qquad \given \qquad &f_0,f_1,f_2,f_3 \text{ of degrees } 0,k,2k,3k,\\
	\qquad \find \qquad &q_1,\ldots,q_m \in \R[X]_{k}, \nonumber \\
	\qquad &\lambda_1,\ldots,\lambda_m \in \R_{\ge 0}, \nonumber \\
	\qquad \subjto \qquad &f_d = \sum_{i = 1}^{m} \lambda_i q_i^d, \qquad d =0,1,2,3.\nonumber
\end{align}
Let us call $ f_0, f_1, f_2, f_3 $ the \emph{power sums}, $ q_1,\ldots,q_m $ the \emph{addends} and $ \lambda_1,\ldots,\lambda_m $ the (nonnegative) \emph{weights}. Throughout, $ X = (X_1,\ldots,X_n) $ are polynomial variables and $ n, m, k\in \N $ are positive integers. For the scope of this paper, we limit our attention to real POF decompositions. The recovery task is not necessarily well-posed, since a given form $ f $ might have various decompositions. However, if $ m $ is not too large, e.g., if $ m = \mathcal{O}(n^{(d-1)k}) $, then a general form of $ k $-Waring rank $ m $ has a unique decomposition, cf. \cite{Blomenhofer_Casarotti_Michalek_Oneto_2022}. 

The special case $ k = 1 $ relates to a plethora of important problems, e.g.: {atom reconstruction} of finitely supported measures, mixtures of Gaussians with \emph{identical} covariance matrices and to {symmetric tensor} decomposition  \cite{Taveira_Blomenhofer_Thesis}. The case $ k = 2 $ of quadratic forms has to do with mixtures of \emph{centered} Gaussians. Therefore, it has Machine Learning applications, e.g., for learning a union of subspaces. These connections are explained in \Cref{sec:applications}. For mixtures of centered Gaussians, third order powers-of-quadratics decomposition yields the first case where nontrivial recovery results are achievable. This case is a special focus of the present paper, although our algorithm works for all values of $ k $. 

\subsection{Overview of Contributions and main results}

This paper proposes and analyzes an algorithm to recover third order POF decompositions, as stated in \eqref{eq:pof-statement-into}. A proof-of-concept implementation of the algorithm in \texttt{Julia} can be found on GitHub, see \cite{Taveira_Blomenhofer_Pofdecomp_Github_2023}.  \Cref{algo:pof-recovery-general} is based on semidefinite programming and is ``efficient'' in the sense that, aside from its calls to the SDP solver, which is treated as a blackbox, it only uses basic linear-algebraic operations on polynomially-sized quantities constructed from the input. In other words, one could say that it is efficient up to the automatizability of semidefinite programming, which is still not completely understood, cf. \cite{Odonnell_2017}. Any numerical troubles, such as condition,  are also ignored. E.g., for the sake of readability, we will write ``$ \lambda_1 > 0 $'' rather than requiring $ \lambda_1 $ to be sufficiently bounded away from zero. 

When it succeeds, the algorithm will also produce a proof of uniqueness of the decomposition. Therefore, it implies an algorithmic uniqueness result, \Cref{thm:main-result}, which is at the core of this paper. The conditions of \Cref{thm:main-result} can be explicitly described in terms of just the second power sum $ f_2 $ of degree $ 2k $ and checked \emph{before} computing the decomposition. One basic tool is to associate a subspace of degree-$ k $ forms to the second order power sum $ f_2 $, which will be called the \emph{Sum of Squares support} of $ f_2 $. It consists of all polynomials contributing to \emph{some} Sum-of-Squares decomposition of $ f_2 $:
\begin{align}\label{eq:intro-sosupp-def}
	\sosupp f_2 = \{p\in \R[X]_{k} \mid \exists \lambda \in \R_{> 0}\colon f_2-\lambda p^2 \text{ is a sum of squares}  \}.
\end{align}
The Sum of Squares support will be explained in detail in \Cref{sec:prelims}. For now, it was just introduced in order to state the main result. 

\begin{theorem}[Cf. \Cref{thm:main-result}]\label{thm:main-result-intro}  
	Let $ k\in \N $ and let $ f_2, f_3 $ be forms of degree $ 2k $ and $ 3k $, respectively. Denote $ U := \sosupp f_2 $ and $ N:=\dim U $. Assume that the space of threefold products $ \{u\cdot v\cdot w \mid u,v,w\in U\}$ of the polynomials in $ U $ has dimension $ \binom{N+2}{3} $.\footnote{In other words, there are no algebraic relations between linearly independent elements of $ U $ of degree at most $ 3 $.}
	Then $ f_2 $ and $ f_3 $ have at most one joint POF decomposition
	\begin{align}
		f_d = \sum_{i = 1}^{m} \lambda_i q_i^d, \qquad d=2,3
	\end{align}
	with positive weights $ \lambda_1,\ldots,\lambda_{m} > 0 $, $ m\in \N $ and linearly independent $ q_1,\ldots,q_m $. Furthermore, if such a decomposition exists, then \Cref{algo:pof-recovery-general} computes it efficiently and it is the unique minimum rank POF decomposition of $ (f_2, f_3) $. 
\end{theorem}

Unlike Waring decompositions of order $ 3 $ or higher, Sum-of-Squares representations are highly non-unique. Indeed, consider a Sum-of-Squares representation $ f = p_1^2  + \ldots + p_N^2 $ of some form $ f $ and any orthogonal matrix $ A\in \R^{N\times N} $. Write $ p:=(p_1,\ldots,p_N) $. Then clearly also the entries of $ s := Ap $ give a Sum-of-Squares representation of $ f $, since 
\begin{align}
	s_1^2  + \ldots +  s_N^2 = s^{T}s = p^{T}A^{T}A p = p^{T}p = p_1^2  + \ldots + p_N^2 = f.
\end{align}
Nevertheless, we say that a form $ f\in \R[X]_{2k} $ is \emph{uniquely Sum-of-Squares representable}, if there is only one Sum-of-Squares representation modulo orthogonal transformations.  

\Cref{sec:decomposition-algorithms} derives and explains the idea behind \Cref{algo:pof-recovery-general} and states the main result, while \Cref{sec:interpretation} offers a detailed discussion of the conditions. Here, we just highlight a few corollaries for instances constructed from \emph{general} addends $ q_1,\ldots,q_m $. One important quantity will be the number $ \beta_{k}(n) $ introduced below. 

\begin{defi}
	Let $ n\in \N$. Then we denote by $ \beta_{k}(n) $ the maximum number $ m $ of $ k $-forms $ q_1,\ldots,q_m $ that satisfy one of the following equivalent conditions
	\begin{enumerate}
		\item There are no algebraic relations of $ q_1,\ldots,q_m $ of degree at most $ 3 $. 
		\item There are no homogeneous algebraic relations of $ q_1,\ldots,q_m $ of degree $ 3 $.  
		\item $ \dim \R[q_1,\ldots,q_m]_{3k} = \binom{m+2}{3} $. 
	\end{enumerate}
\end{defi}

It holds that $ \beta_{k}(n) \ge n+1 $ for all $ k \ge 2, n\ge 2 $ and, e.g., that $ \beta_{2}(n)=\Theta(n^2) $, see \Cref{sec:real-geometry}. For power sums constructed from at most $ m\leq  \beta_{k}(n) $ \emph{general} addends, the conditions of \Cref{thm:main-result-intro} are satisfied, if $ f_2 $ is uniquely Sum-of-Squares representable. See \Cref{sec:real-geometry} and, in particular,  \Cref{prop:unique-sos-representations} for a detailed discussion. This yields the following simplification.

\begin{corollary}\label{cor:pof-recovery-intro-gram} 
	For any $ n, m, k\in \N $ with $ m\le \beta_{k}(n) $, there is an efficient algorithm for the following problem: If  $ \lambda_1,\ldots,\lambda_m \in \R_{>0} $ and $ q_1,\ldots, q_m \in \R[X]_{k} $ are general forms such that $ f_2 := \sum_{i = 1}^{m} \lambda_i q_i^2 $ is uniquely Sum-of-Squares representable, compute the set $ \{(q_1,\lambda_1),\ldots, (q_m,\lambda_m)\} $ from inputs $ f_2 $ and $ f_3 := \lambda_1 q_1^3 + \ldots + \lambda_m q_m^3 $.
\end{corollary}

\noindent
For $ m < n $, it is in addition possible to give geometric criteria. 

\begin{corollary}\label{cor:pof-recovery-intro-real-variety} 
	For any $ n, m\in \N $ with $ m\leq n-2 $, there is an efficient algorithm for the following problem: If $ q_1,\ldots, q_m \in \R[X]_{k} $ are general forms such that their real variety $ V_{\R}(q_1,\ldots,q_m) $ contains a nonzero point and $ \lambda_1,\ldots,\lambda_m $ are positive, compute the set $ \{(q_1,\lambda_1),\ldots, (q_m,\lambda_m)\} $ from inputs $ f_2 := \lambda_1 q_1^2 + \ldots + \lambda_m q_m^2 $ and $ f_3 := \lambda_1 q_1^3 + \ldots + \lambda_m q_m^3 $. 
\end{corollary}

\noindent
The case $ m=n-1 $ is special: 
\begin{corollary}\label{cor:pof-recovery-intro-real-variety-m=n-1} 
	For any $ n\in \N $, there is an efficient algorithm for the following problem: If $ q_1,\ldots, q_{n-1} \in \R[X]_{k} $ are general forms such that all the finitely many lines of the variety $ V(q_1,\ldots,q_{n-1}) $ are real and $ \lambda_1,\ldots,\lambda_m $ are positive, compute the set $ \{(q_1,\lambda_1),\ldots, (q_m,\lambda_m)\} $ from inputs $ f_2 := \lambda_1 q_1^2 + \ldots + \lambda_m q_m^2 $ and $ f_3 := \lambda_1 q_1^3 + \ldots + \lambda_m q_m^3 $.  
\end{corollary}

In \Cref{app:typical-regions-gram-spec}, it is proven that the conditions of both \Cref{cor:pof-recovery-intro-real-variety} and \Cref{cor:pof-recovery-intro-real-variety-m=n-1} are satisfied for \emph{typical} choices of $ q_1,\ldots,q_m $, i.e. on a Euclidean open subset of $ \R[X]_{k}^m $. However, note that the geometric arguments fail for $ m \gg n $, whereas one has evidence to believe that \Cref{algo:pof-recovery-general} can also decompose some instances of quadratic rank $ m=\Theta(n^2) $, cf. \Cref{conj:pof-recovery-intro-gram}. Therefore it is not clear whether \Cref{cor:pof-recovery-intro-gram} extends beyond $ m=n-1 $, but numerical evidence from \Cref{sec:trace-free-quadratics-study} strongly suggests so. Unique Sum-of-Squares representability appears to be in general not well-understood and \Cref{cor:pof-recovery-intro-gram} gives further motivation to understand it better. The recovery result from \Cref{thm:main-result} has consequences for certain Machine Learning problems, two of which we highlight in the following section. 

\subsection{Learning parameters of centered Gaussian mixtures}

The parameter estimation problem for Gaussian mixtures has a rich history, dating back to Pearson \cite{Pearson_1900}. It has now been studied over more than a century in all kinds of flavours e.g. from the perspective of computer science (\cite{Dasgupta_1999a},\cite{Sanjeev_Kannan_2001a},\cite{Dasgupta_Schulman_2007a},\cite{Moitra_Valiant_2010},\cite{Kalai_Moitra_Valiant_2010},\cite{Hsu_Kakade_2013a},\cite{Anderson_Belkin_Goyal_Rademacher_Voss_2014a},\cite{Regev_Vijayaraghavan_2017a},\cite{Liu_Moitra_2021}), algebraic geometry (\cite{Amendola_Faugere_Sturmfels_2016},\cite{Amendola_Ranestad_Sturmfels_2018}), moment problems (\cite{Curto_DiDio_2022},\cite{DiDio_2023}) and applications (\cite{Reynolds_Rose_1995a},\cite{Permuter_Francos_Jermyn_2003a}). A mixture of $ m $ \emph{centered} Gaussians has a degree-$ 2d $ moment form (cf. \Cref{sec:gmm}) proportional to 
\begin{align}
	\sum_{i = 1}^{m} \lambda_i q_i^d,
\end{align}
where $ \lambda_1,\ldots,\lambda_m \in \R_{\ge 0} $ are the mixing weights (summing up to $ 1 $) and $ q_1,\ldots,q_m $ are positive (semi)definite quadratic forms $ q_i = X^{T}\Sigma_i X $, where $ \Sigma_i $ is the covariance matrix of the $ i $-th centered Gaussian. Thus, there is a straightforward connection between the parameter estimation problem for mixtures of centered Gaussians from their moments on one side and decompositions as powers of quadratic forms on the other side. At first glance, \Cref{thm:main-result-intro} does not fare well together with the setting of Gaussian mixtures, since if one of the forms $ q_1,\ldots,q_m $ is positive definite, then $ \sum_{i = 1}^{m} \lambda_i q_i^2 $ will never be uniquely Sum-of-Squares representable. With some slight adaptations, it is possible to prove a recovery result for typical instances of Gaussian mixtures. This is done in \Cref{sec:applications} and highlighted here: 

\begin{corollary}\label{cor:cgmm-recovery-intro} 
	For any $ n\in \N $, $ m \in \{1,\ldots, n-1\} $, there is a Euclidean open subset $ \mathcal{U} $ of $ \R[X_1,\ldots,X_n]_{2}^m $ and an efficient algorithm for the following problem:  If $ Y $ is a mixture of $ m $ centered Gaussian random variables with general positive definite covariance forms $ (q_1,\ldots,q_m) \in \mathcal{U} $ and positive mixing weights $ \lambda_1,\ldots,\lambda_m $, compute the set of parameters $ \{(q_1, \lambda_1),\ldots, (q_m, \lambda_m)\} $ from the moments $ \mathcal{M}_{\leq 6}(Y) $ of $ Y $ of degree at most $ 6 $. 
\end{corollary}

\subsection{Learning unions of subspaces}

A special type of Gaussian mixture distributions can be used as a model for subspace learning. Here, data is assumed to be normally distributed on either of the $ r $-dimensional subspaces $ U_1,\ldots,U_m $ and the task is to find bases for the subspaces $ U_1,\ldots,U_m $ from samples of the mixture distribution as input. The main difference to a general Gaussian mixture instance from above is that the forms $ q_1,\ldots,q_m $ corresponding to the subspaces $ U_1,\ldots,U_m $ will not have full rank. 

We highlight this special application, since it is a case where one needs uniqueness not for \emph{general} forms, but for forms that are general \emph{within the class of fixed-rank quadratic forms}. We are not aware of any decomposition result applicable for this case, since the previous work \cite{Bafna_Hsieh_Kothari_Xu_2022}, \cite{Ge_Huang_Kakade_2015} is based on a probabilistic analysis and thus implicitly assumes full-rank quadratics.

\begin{corollary}\label{cor:unionspaces-recovery-intro} 
	For any $ n, r\in \N_{\ge 3} $, $m \leq n-1$, there is a Euclidean open subset $ \mathcal{U} $ of the problem parameters\footnote{The parameters are the subspaces together with the means and covariances of the Gaussians. } and an efficient algorithm for the following problem: If $ Y_1,\ldots,Y_m $ are normally distributed random variables on $ r $-dimensional subspaces $ U_1,\ldots,U_m $ and $ \lambda\in \R_{> 0}^m $ with $ \sum_{i = 1}^{m} \lambda_i = 1 $, compute bases for the subspaces $ U_1,\ldots,U_m $ from the moments of $ \lambda_1 Y_1 \oplus \ldots \oplus \lambda_m Y_m $ of degree at most $ 6 $. 
\end{corollary}

\subsection{Relevance of results}
The aim of this work is to get tighter, algorithmic rank lower bounds for third-order powers-of-forms decomposition. In a typical real case, with \Cref{cor:pof-recovery-intro-real-variety-m=n-1}, we improve the rank threshold for efficient recovery from $ m\le \mathcal{O}(\frac{n}{\log(n)^2}) $ (due to \cite{Bafna_Hsieh_Kothari_Xu_2022}) or $\Omega(1)\le m\le \mathcal{O}(\sqrt{n}) $ (due to \cite{Ge_Huang_Kakade_2015}) to $ 1\le m\le n-1 $. Compared to the previous results, this gives an improvement of the asymptotic order \emph{and} the constant factors, with a much simpler proof. In addition, \Cref{algo:pof-recovery-general} implicitly produces a proof of uniqueness of the minimum rank decomposition for any concrete instance where it succeeds. 

Analysis beyond the case $ m = n-1 $ is more difficult, since one may not rely on geometric arguments any more. However, there is significant reason to hope that \Cref{algo:pof-recovery-general} can  decompose instances of rank $ m=\Theta(n^2) $. Indeed, $ m=\Theta(n^2) $ general quadratics $ q_1,\ldots,q_m $ do not satisfy any algebraic relations of degree $ 3 $, which is implicitly shown by Bafna, Hsieh, Kothari and Xu \cite[Section 6.4]{Bafna_Hsieh_Kothari_Xu_2022}. The big open question is whether there are also instances of $ m=\Theta(n^2) $ quadratics that are uniquely Sum-of-Squares representable, and if these sets have nonempty intersection. Numerical findings suggest this is the case, cf. \Cref{conj:pof-recovery-intro-gram}. The threshold $ m\in \mathcal{O}(n) $ is thus likely not an actual algorithmic boundary, see \Cref{sec:trace-free-quadratics-study}.

\subsubsection{Acknowledgements}
I wish to thank Pravesh Kothari, who encouraged me to work on powers-of-forms decompositions and pointed me towards the work of Ankit Garg \cite{Garg_Kayal_Saha_2020}. I also wish to thank Monique Laurent for the rich feedback provided to this article, in particular for catching several mistakes.  I further wish to thank Greg Blekhermann and João Gouveia for the suggestion to study trace-free quadratic forms, Julian Vill and Claus Scheiderer for sharing their expertise on Gram Spectrahedra and to Simon Telen for being a good listener. Part of this work was completed while the author was supported by the Dutch Scientific Council (NWO) grant OCENW.GROOT.2019.015 (OPTIMAL). 

\subsubsection{Disclosure}
The main ideas of this paper were published first as part of my doctoral thesis \cite{Taveira_Blomenhofer_Thesis} at Universität Konstanz. Some formulations might therefore overlap. However, this paper significantly elaborates on the ideas that were present in \cite{Taveira_Blomenhofer_Thesis}.

%% file: prelims.tex
\raggedbottom
\paragraph{Notation}

Let us write $ \N = \{1,2,3,\ldots \} $ for the set of natural numbers and $ \N_0 $ for $ \N \cup \{0\} $. This paper concerns POF decompositions over the real field $ \R $, but we might occasionally mention some results that hold over the complex numbers $ \C $. For $ K\in \{\R, \C\} $, we endow any finite dimensional $ K $-vector space $ U $ with the $ K $-Zariski topology. The varieties considered in this paper are closed affine or projective varieties. Closed affine varieties are subsets of $ U $ that can be written as the feasible set $ V(q_1,\ldots,q_m) $ of a  system of polynomial equations
\begin{align}
	q_1(x) = 0, \ldots, q_m(x) = 0, \quad (x\in U).
\end{align}
The space of linear functionals from $ U $ to $ \R $ is denoted $ U^{\dual} $ and called the \emph{dual space} of $ U $. Algebraic unknowns will be denoted by capital letters. In particular, for $ U = K^n $, it is by default assumed that the unknowns are $ X = (X_1,\ldots,X_n) $ and the polynomial ring is denoted $ K[X] $. Note that $ p\in K[X] $ denotes a polynomial, whereas $ p(x) $ denotes the evaluation of $ p $ in some point $ x\in K^n $. As one exception, when talking about algebraic relations of some polynomials $ q_1,\ldots,q_m\in K[X] $, let us denote their \emph{ideal of relations} 
\begin{align}
	\irel(q_1,\ldots,q_m) = \{f\in K[Y] \mid f(q_1,\ldots,q_m) = 0\}
\end{align}
in some separate set of unknowns $ Y = (Y_1,\ldots,Y_m) $, to avoid confusion. For some graded $ K $-algebra $ R $, $ R_{k} $ denotes the $ k $-th graded component of $ R $ and $ R_{\leq k} := R_{0} \oplus \ldots \oplus R_{k}$ denotes the part of grade at most $ k $. Quotients of polynomial rings by homogeneous ideals will naturally inherit the grading by the degree. However, for a subalgebra $ K[q_1,\ldots,q_m] \subseteq K[X_1,\ldots,X_n]$ generated by some $ k $-forms $ q_1,\ldots,q_m $, we will often deviate from the canonical grading by the degree and instead grade $ K[q_1,\ldots,q_m] $ by $ \frac{1}{k} $ times the degree, for technical reasons. 

\medskip
The reader is assumed to have some basic familiarity with convex geometry (e.g., the notions of convex cones, faces, relative interior, conic duality) and with algebraic geometry (e.g., Bertini's theorem). For the background knowledge, cf. the books of Barvinok \cite{Barvinok_2002} for convex geometry and Hartshorne \cite{hartshorne2013algebraic} for algebraic geometry. For a convex cone $ C\subseteq U $ in some $ \R $-vector space $ U $, the \emph{dual cone} of $C$ is denoted as 
\begin{align}
	C^{\ast} := \{L\in U^{\dual} \mid \forall u\in U\colon L(u) \ge 0 \} \subseteq U^{\dual}.
\end{align}
In the special case where $ C $ is even a subspace of $ U $, it holds
\begin{align}
	C^{\ast} = \{L\in U^{\dual} \mid \forall u\in U\colon L(u) = 0 \}.
\end{align}
Thus, $ C^{\ast} $ is then a subspace of $ U^{\dual} $, which is called the \emph{conormal space} of $ C $. It is commonly denoted $ C^{\perp} $ rather than $ C^{\ast} $. If $ R $ is a commutative, graded $ \R $-algebra with graded components $ R_0, R_1, R_2,\ldots $, then for $ k\in \N_0 $, we denote by 
\begin{align}
		\Sigma_{R, 2k} = \left\{f\in R_{2k} \mid \exists N \in \N_0, p_1,\ldots, p_N \in R_{k}\colon f = \sum_{i=1}^N p_i^2  \right\}.
\end{align}
the \emph{homogeneous Sums-of-Squares cone of $ R $} in degree $ 2k $. If $ R = \R[X] $ is the polynomial ring, we simply write $ \Sigma_{2k} $, suppressing the dependency on the variables. For a homogeneous ideal $ I \subseteq R $, we denote by $ I_k $ the degree-$ k $ component of $ I $, i.e., $ I_k = I\cap R_k $. 

\subsection{Gram Spectrahedra} 
A form $ f \in \R[X] $ is a \emph{sum of squares} if there exist $ N\in \N_0 $ and forms $ q_1,\ldots,q_N \in \R[X] $ such that 
\begin{align}\label{eq:sos-representation}
	f = \sum_{i = 1}^{N} q_i^2.
\end{align}
The right hand side of \eqref{eq:sos-representation} is called a \emph{Sum-of-Squares representation}. For any orthogonal transformation $ A \in \R^{N\times N}$, both $ q = (q_1,\ldots,q_N)^{T} $ and $ A(q_1,\ldots,q_N)^{T} $ represent the same polynomial $ f $. Let us denote by $ [X]_{k} = (X^{\alpha})_{|\alpha| = k} $ the vector of monomials of degree $ k $. Then, any polynomial $ p\in \R[X]_{k} $ can be written as $ p = c_p^{T}[X]_{k} $ for some real \emph{coefficient vector} $ c_p = (c_{p, \alpha})_{|\alpha| = k} $. This allows to write Sum-of-Squares representations such as \eqref{eq:sos-representation} via \emph{Gram matrices} 
\begin{align}\label{eq:gram-matrix-representation}
	f = [X]_{k}^{T} \left(\sum_{i = 1}^{N} c_{q_i}c_{q_i}^{T}\right)[X]_{k} = [X]_{k}^{T}G(q)[X]_{k}.
\end{align}
Here, we denote $ G(q) := \sum_{i = 1}^{N} c_{q_i}c_{q_i}^{T} $ for the positive semidefinite (psd) Gram matrix of the Sum-of-Squares representation $ f = q^{T}q $. Let us write $ G\succeq 0 $ to denote that some (symmetric) matrix $ G $ is psd. It turns out that any matrix representation $ f = [X]_{k}^{T}G[X]_{k} $, where $ G\succeq 0 $, corresponds to a class of Sum-of-Squares representations modulo orthogonal transformations. The convex set 
\begin{align}
	\gram(f) := \{G \succeq 0 \mid [X]_{k}^{T}G[X]_{k} = f \}
\end{align}
is called the \emph{Gram spectrahedron} of $ f $. 
Let us collect some basic properties. 

\begin{prop}\label{prop:gram-spec}
	Let $ k\in \N $, $ f \in \Sigma_{2k} $. 
		\begin{enumerate}[(a)]
		\item Every face $ F $ of $ \gram(f) $ has an associated subspace $ U_F $ such that 
		\begin{align*}
			F = \{G\in \gram(f) \mid \im G \subseteq U_F\}
		\end{align*}
		and such that equality $ \im G = U_F $ holds for \emph{all} points in the relative interior of $ F $. We interpret $ U_F $ as a subspace of $ \R[X]_{k} $, by sending $ c\in U_F $ to $[X]_{k}^{T} c $. 
		\item A relative interior point of $ F $ corresponds to a class of Sum-of-Squares representations of $ f $ (modulo orthogonal transformations) of length $ \dim U_F $. 
		\item A linear subspace $ U $ of $ \R[X]_{k} $ is called \emph{facial} for $ \gram(f) $, if there exists some $ G $ in $ \gram(f) $ such that $ \im G = U $. 
		\item If $ F' \subsetneq F $ is a proper subface, then $ \dim U_{F'} < \dim U_{F} $. 
		\item The set 
		\begin{align*}
			\sosupp f = \{p\in \R[X]_{k} \mid \exists \lambda \in \R_{> 0}\colon f-\lambda p^2 \text{ is a sum of squares}  \}
		\end{align*}
		of all polynomials contributing to some Sum-of-Squares decomposition of $ f $ is a subspace of $ \R[X]_{k} $. 
		\item The sum of facial subspaces is facial. In particular, there exists a largest facial subspace $ U_{\gram(f)} $ of $ \gram(f) $ and this subspace equals $ \sosupp f $. 
	\end{enumerate} 
\end{prop}
\begin{proof}
	Cf. the work of Ramana and Goldman \cite{Ramana_Goldman_1995} on the facial structure of (arbitrary) spectrahedra. This formulation loosely follows Scheiderer \cite[Section 2]{Scheiderer_2022}.
\end{proof}

Recall that for a point $ x $ in some convex set $ C $, the \emph{supporting face} $ \suppf x $ of $ x $ is defined to be the minimal face of $ C $ containing $ x $. For a face $ F $ of $ C $ it holds $ x\in \relint F $, if and only if $ F $ is the supporting face of $ x $. 

\subsection{Powers-of-forms decomposition}\label{sec:pof-decomposition-setting} Throughout, we will consider third order powers-of-forms decomposition, as introduced in \eqref{eq:pof-statement-into}. Let us restate the standard setting of third order POF decomposition used in this paper:  

\begin{align}\label{eq:pof-statement-prelims}
	(\pof)_{f, m, k} \qquad \given \qquad &f_0,f_1,f_2,f_3 \text{ of degrees } 0,k,2k,3k,\\
	\qquad \find \qquad &q_1,\ldots,q_m \in \R[X]_{k}, \nonumber \\
				 \qquad &\lambda_1,\ldots,\lambda_m \in \R_{\ge 0}, \nonumber \\
	\qquad \subjto \qquad &f_d = \sum_{i = 1}^{m} \lambda_i q_i^d, \qquad d =0,1,2,3.\nonumber
\end{align}
Note that in order to recover both the addends and the weights, it is necessary to use power sums of at least two different orders. Our main results, \Cref{thm:main-result} and \Cref{algo:pof-recovery-general}, are ``minimal'' in the sense that they only make use of the power sums $ f_2 $ and $ f_3 $. In some applications, it is canonical and useful to have $ f_1 $ as well. This is explained in \Cref{sec:applications}. 

\begin{defi}\label{def:waring-rank-and-length}  
	Let $ k, d\in \N $ and $ f\in \R[X]_{dk} $. There exists a smallest number $ m\in \N $ such that $ f $ has a $ k $-\emph{Waring decomposition} of \emph{rank} $ m $, i.e., a decomposition 
	\begin{align}
		f = \sum_{i = 1}^{m} \sigma_i q_i^d
	\end{align}	
	of $ f $ as a signed sum of $ m $ $ d $-th powers of $ k $-forms $ q_1,\ldots,q_m $, with \emph{signs} $ \sigma_1,\ldots,\sigma_m\in \{\pm 1\} $. This $ m $ is called the (real) $ k $-\emph{Waring rank} of $ f $. 
	For odd $ d $, the signs can be omitted. For even $ d $, let us define the $ k $-\emph{length} of $ f $ as the smallest number $ m\in \N\cup \{\infty\} $ such that $ f $ has a $ k $-Waring decomposition of length $ m $, with all signs being positive. We denote it by $ \length_{k} f $ and understand it as $ \infty $, whenever there is no such decomposition. In the case $ d = 2 $, $ \length f := \length_{2} f $ is called the \emph{Sum-of-Squares length}, or simply the \emph{length} of $ f $. The $ k $-Waring rank of a \emph{generic} $ kd $-form in $ n $ variables is denoted $ \rk_{k}^{\circ}(n, kd) $. 
\end{defi}

\subsection{Powers of linear forms}

The case $ k = 1 $ of powers of linear forms is comparatively well-understood. A classical uniqueness result for cubic forms of very low rank is known due to Jennrich (via Harshman \cite{Harshman_1970}). There exist efficient methods to extract the linear forms, c.f. Anandkumar, Ge, Hsu, Kakade and Telgarsky \cite{Anandkumar_Ge_Hsu_Kakade_Telgarsky_2012}. 

\begin{theorem}\label{thm:undercomplete-waringdecomp} (cf. e.g. \cite{Leurgans_Ross_Abel_1993},\cite{Anandkumar_Ge_Hsu_Kakade_Telgarsky_2012})
	There exists an algorithm that, on input $ n \in \N$ and forms $ f_2 $, $ f_3 $ of degrees $ 2 $ and $ 3 $, respectively, computes the solution to the following problem: If $ f_2, f_3 $ have a POF decomposition
	\begin{align}\label{eq:undercomplete-waringdecomp-1}
		f_d = \sum_{i = 1}^{m} \lambda_i \ell_i^d
	\end{align} 
	such that $ \ell_{1},\ldots,\ell_{m} $ are linearly independent and $ \lambda_1,\ldots,\lambda_m\in \R\setminus\{0\} $, then compute $ {(\ell_{1},\lambda_1),\ldots,(\ell_{m},\lambda_m)} $. Under these conditions, \eqref{eq:undercomplete-waringdecomp-1} is the unique minimum rank POF decomposition of $ (f_2, f_3) $ and the only POF decomposition with linearly independent addends. 
\end{theorem}
\begin{proof}
	The proof is deferred to \Cref{app:polinearforms}.
\end{proof}

%% file: pof-algo-basic.tex
\subsection{Overview of ideas and techniques}\label{sec:decomposition-algorithms}

The main result of this paper is a recovery algorithm for the addends and weights of a third order powers-of-forms decomposition, as described in \Cref{sec:pof-decomposition-setting}. It is simultaneously also an algorithmic proof of uniqueness of the decomposition, and can thus be seen as a generalization of the classical result, \Cref{thm:undercomplete-waringdecomp}, commonly attributed to Jennrich \cite{Harshman_1970}. Some of the conditions impose implicit constraints on the rank of the POF decomposition. \Cref{sec:interpretation} discusses these implications in detail and \Cref{sec:trace-free-quadratics-study} proves uniqueness of the POF decomposition for some concrete examples. 

The algorithm combines two simple ideas: First, we aim to recover the space $ \langle q_1,\ldots,q_m \rangle $ spanned by the addends. Note that this is a trivial task for $ k=1 $, but for $ k\ge 2 $ it is not. Then, given a basis $ u_1,\ldots,u_m $ for the space $ \langle q_1,\ldots,q_m \rangle $, note that sometimes it is possible to reduce the $ k $-Waring decomposition problem to a $ 1 $-Waring decomposition problem. Let us start by explaining this second idea:

\subsubsection{Second idea: Reduction to $ k = 1 $} $ \R[q_1,\ldots, q_m] $ is an algebra graded by $ \frac{1}{k} $ times the degree and the kernel of the graded algebra homomorphism
\begin{align}
	\varphi\colon \R[Y_1,\ldots, Y_m] \to \R[q_1,\ldots, q_m], Y_1\mapsto u_1,\ldots, Y_m\mapsto u_m
\end{align}
is the ideal $ \irel(u_1,\ldots,u_m) $ of algebraic relations of $ u_1,\ldots,u_m $, which, via a change of coordinates, translates to the ideal of relations of $ q_1,\ldots,q_m $. If $ \irel(q_1,\ldots,q_m) $ does not contain forms of degree at most $ 3 $, then the restriction $ \varphi_{\leq 3} $ of $ \varphi $ to $ \R[Y_1,\ldots,Y_m]_{\leq 3} $ is an invertible linear map onto its image $ \R[q_1,\ldots,q_m]_{\leq 3} $. The inverse map $ \varphi_{\leq 3}^{-1} $ must map the $ k $-forms $ q_1,\ldots,q_m $ in $ X_1,\ldots,X_n $ to some linear forms $ \ell_{1},\ldots,\ell_{m} $ in $ Y_1,\ldots,Y_m $. One easily sees that for $ d\in \{1,2,3\} $:
\begin{align}
	g_d := \varphi_{\leq 3}^{-1}(f_d) = \sum_{i = 1}^{m} \ell_i^d
\end{align}
admit a joint decomposition as powers of \emph{linear} forms. From that, a classical algorithm based on eigenvalue decomposition can be used, which is described in \Cref{thm:undercomplete-waringdecomp} and \Cref{app:polinearforms}. Note that the inverse of $ \varphi_{\leq 3} $ can be computed: Since the $ d $-fold products $ (u^{\alpha})_{|\alpha|=d} $ of entries of $ u $ form a basis of $ \R[U]_{d} $ for $ d = 1,2,3 $, there exist unique coefficients $ (c_{\alpha})_{|\alpha|=d} $ such that 
\begin{align}
	f_d = \sum_{|\alpha|=d} c_{\alpha} u^{\alpha}, \quad (d = 1,2,3),
\end{align}
which can be obtained by linear system solving. Then, $ g_d = \sum_{|\alpha|=d} c_{\alpha} Y^{\alpha} $. 

\medskip
It is easy to see that the invertibility condition is a generic property at least as long as $ m \leq n+1 $, but e.g. for quadratics $ q_1,\ldots,q_m\in \R[X]_{2} $, it also holds for some $ m=\Theta(n^2) $. This is discussed in \Cref{sec:first-condition}. 

\subsubsection{First idea: Space recovery}\label{sec:idea-space-recovery} To recover a basis of the space $ \langle q_1,\ldots,q_m \rangle $, we make heuristic use of the Gram spectrahedron. 
By \Cref{prop:gram-spec}, there is a superspace $ U_F $ of $ \langle q_1,\ldots,q_m \rangle $ associated with every face $ F $ of $ \gram(f_2) $ containing the Gram matrix $ G(q) $ of the representation $ f_2 = \sum_{i = 1}^{m} q_i^{2} $. In particular, one of these subspaces is equal to $ \langle q_1,\ldots,q_m \rangle $. It corresponds to the supporting face of $ G(q) $. 

It is not clear whether the faces containing $ G(q) $ are accessible to us from input $ f_2, f_3 $.  Fortunately, there are many cases where $ \gram(f_2) $ has a particularly simple structure. The simplest possible case is when $ f_2 $ is uniquely Sum-of-Squares representable. Then, $ \gram(f_2) = \{G(q)\} $ is a singleton. It suffices to compute the unique Gram matrix $ G = G(q) $ of $ f_2 $ and its image will give the space of $ q_1,\ldots,q_m $. The second simplest case is when $ G(q) \in \relint \gram(f_2) $: Then, while there might be several nonequivalent Sum-of-Squares representations, the space $ \langle q_1,\ldots,q_m \rangle $ is still accessible to us, since we can compute a relative interior point of $ \gram(f_2) $ with an interior point solver for SDPs. 
If $ f_2 $ is constructed from (not too many) \emph{generic} addends $ q_1,\ldots,q_m $, then $ G(q) \in \relint \gram(f_2) $ is in fact equivalent to $ \gram(f_2) = \{G(q)\} $, cf. \Cref{prop:unique-sos-representations}. In both these cases, it holds $ \sosupp f_2 = \langle q_1,\ldots,q_m \rangle $ by \Cref{prop:gram-spec}(f). In all other cases, note that one may still take the potentially larger space $ \sosupp f_2 $ as an ``upper approximation'' for the space $ \langle q_1,\ldots,q_m \rangle $, and hope for the best.  

To justify that the approach is reasonable, we will show that there are sufficiently many choices of $ q_1,\ldots,q_m $, such that their second order power sum $ f_2 $ is uniquely Sum-of-Squares representable. This is done in \Cref{sec:interpretation}. 


\subsection{Algorithms} 

The procedure to recover the POF decomposition is described in \Cref{algo:pof-recovery-general}. The following theorem is a uniqueness result for POF decomposition, derived as a consequence. Note that for the first read, it is instructive to have the case in mind where $ f_2 = \sum_{i = 1}^{m} \lambda_i q_i^2 $ is uniquely Sum-of-Squares representable. In this case, $ N = m $ in both \Cref{thm:main-result} and \Cref{algo:pof-recovery-general}, and it holds $ U = \langle q_1,\ldots,q_m \rangle $. The condition $ \dim  \R[U]_{3k} = \binom{m+2}{3} $ is then equivalent to $ \irel(q_1,\ldots,q_m)_{3} = \{0\} $, which, according to \cite[Section 6.4]{Bafna_Hsieh_Kothari_Xu_2022}, is satisfied for generic $ q_1,\ldots,q_m $ with $ m=\Theta(n^2) $. Note that a tentative implementation of \Cref{algo:pof-recovery-general}, with an example Julia notebook, can be found on GitHub. See \cite{Taveira_Blomenhofer_Pofdecomp_Github_2022}.

\begin{theorem}\label{thm:main-result}
	Let $ k\in \N $ and let $ f_2, f_3 $ be forms of degree $ 2k $ and $ 3k $, respectively. Denote $ U:= \sosupp f_2 $ and $ N:=\dim U $. Assume that the graded component $ \R[U]_{3k} $ has dimension $ \binom{N+2}{3} $. 
	Then, $ f_2 $ and $ f_3 $ have at most one joint POF decomposition
	\begin{align}
		f_d = \sum_{i = 1}^{m} \lambda_i q_i^d, \qquad d=2,3,
	\end{align}
	with positive weights $ \lambda_1,\ldots,\lambda_{m} > 0 $, $ m\in \N $ and linearly independent $ q_1,\ldots,q_m $. Furthermore, if such a decomposition exists, then \Cref{algo:pof-recovery-general} computes it efficiently and it is the unique minimum rank POF decomposition of $ (f_2, f_3) $. 
\end{theorem}
\begin{proof}
	Assume there are two distinct POF decompositions, the left one of which had linearly independent addends,
	\begin{align}
		\sum_{i = 1}^{m} \lambda_i q_i^d = f_d = \sum_{i = 1}^{m'} \mu_i p_i^d, \quad d = 2,3,
	\end{align}
	with positive weights $ \lambda_i, \mu_i $ and $ m, m' \in \N $. Then by \Cref{prop:gram-spec}, it holds that 
	\begin{align}
		\langle q_1,\ldots,q_{m} \rangle \subseteq U \supseteq \langle p_1,\ldots,p_{m'} \rangle.
	\end{align} 
	The linearly independent system $ q_1,\ldots,q_{m} $ can therefore be extended to a basis $ u = (q_1,\ldots,q_{m},u_{m+1},\ldots,u_N) $ of $ U $. By assumption on the dimension of $ \R[U]_{3}$, there are no algebraic relations of degree $ 3 $ between linearly independent elements of $ U $. Since these relations form a homogeneous ideal, there are also no relations of degree \emph{at most} $ 3 $. 
	For the evaluation map
	\begin{align}
		\varphi\colon \R[Y_1,\ldots,Y_N] \to \R[u_1,\ldots,u_N], Y_i\mapsto u_i \quad (i = 1,\ldots, N),
	\end{align}
	which sends forms of degree $ d\in \N_0 $ to forms of degree $ kd $, the restriction $ \varphi_{\leq 3} $ to $ \R[Y]_{\leq 3} $ is therefore invertible. The inverse $ \varphi_{\leq 3}^{-1} $ maps
	$ f_2 $ and $ f_3 $ to quadratic and cubic forms $ g_2 $ and $ g_3 $, respectively, which admit decompositions
	\begin{align}
		g_d = \sum_{i = 1}^{m} \lambda_i X_i^d = \sum_{i = 1}^{m'} \mu_i \ell_{i}^d, \qquad (d = 2,3)
	\end{align}
	where $ \ell_{i} := \varphi_{\leq 3}^{-1}(p_i) $. However, \Cref{thm:undercomplete-waringdecomp} shows uniqueness of the rank-$ m $ POF decomposition for $ g_2 $ and $ g_3 $. Precisely, \Cref{thm:undercomplete-waringdecomp} implies that $ m'\ge m $ and if $ m=m' $, then, up to reordering, $ \lambda_i = \mu_i $ and $ X_i = \ell_{i} $ for all $ i\in \{1,\ldots,m\} $. Substituting back via $ \varphi $ yields $ q_i = p_i $ for all $ i\in \{1,\ldots,m\} $. 
\end{proof}

\begin{algorithm}[H]
	\caption{Semidefinite algorithm for powers-of-forms decomposition.}
	\label{algo:pof-recovery-general}
	\raggedright
	\textbf{Input: } $ k\in \N $ and forms $ f_2\in \R[X]_{2k}, f_3\in \R[X]_{3k} $.\\
	\textbf{Assumptions: } 
	\begin{enumerate}
		\item $ f_d $ have a joint POF decomposition $ f_d = \sum_{i = 1}^{m} \lambda_i q_i^d $ for $ d = 2, 3 $, where $ q_1,\ldots,q_m \in \R[X]_{k}$ are linearly independent $ k $-forms, $ m\in \N $ and $ \lambda_1,\ldots,\lambda_m \in \R_{>0} $.
		\item For $ U:= \sosupp f_2 $ and $ N:=\dim U $,  $ \R[U]_{3} $ has dimension $ \binom{N+2}{3} $.  
	\end{enumerate}
	\textbf{Output: }  $ \{(q_1,\lambda_1),\ldots,(q_m,\lambda_m)\} $ \\
	\textbf{Procedure: }
	
	\begin{algorithmic}[1]
		\STATE{Use $ f_2 $ and an interior point SDP solver to compute some basis $ u = (u_1,\ldots,u_{N}) $ of $ U $. This can be done by computing some $ G\in \relint \gram(f) $ and a basis of $ \im G $, see \Cref{prop:gram-spec}. Cf. \Cref{app:pointed-sos-cones} for the SDP formulation. }  
		\STATE{The linear system
		\begin{align*}
			& f_d = \sum_{|\alpha| = d} c_{\alpha} u^{\alpha}, \qquad c_{\alpha} \in \R, \qquad (\alpha\in \N_0^N, |\alpha| = d)
		\end{align*}
		has a unique solution $ c_d = (c_\alpha)_{|\alpha| = d} \in \R[Y_1,\ldots,Y_N]_{d} $ for both $ d = 2 $ and $ d = 3 $. Compute it and set
		\begin{align*}
			g_d := \sum_{|\alpha| = d} c_{\alpha} Y^{\alpha}. 
		\end{align*}
		Note that $ g_d = \varphi_{\leq 3}^{-1}(f_d) $, where 
		\begin{align*}
			\varphi_{\leq 3} \colon \R[Y_1,\ldots,Y_N]_{\leq 3} \to \R[u_1,\ldots,u_N]_{\leq 3}, Y_i \mapsto u_i \qquad (i = 1,\ldots,N).
		\end{align*}}
		\STATE{For degree reasons and since the map $ \varphi_{\leq 3} $ is the restriction of an algebra homomorphism, there exist unique linearly independent linear forms $ \ell_{1} = \varphi_{\le 3}^{-1}(q_1),\ldots,\ell_{m}=\varphi_{\le 3}^{-1}(q_m) $ such that $ g_d = \sum_{i = 1}^{m} \lambda_i \ell_{i}^d  $ for $ d = 2,3 $.}
		\STATE{Compute $\{(\ell_{1},\lambda_1),\ldots,(\ell_{m},\lambda_m)\} $ with the algorithm from \Cref{thm:undercomplete-waringdecomp}. }
		\RETURN $ \{(\varphi_{\leq 3}(\ell_{1}),\lambda_1),\ldots,(\varphi_{\leq 3}(\ell_{m}),\lambda_m) \}$. 
	\end{algorithmic}
\end{algorithm}


%% file: interpretation.tex
\subsection{Real Geometry Viewpoint}\label{sec:real-geometry}

Let us try to understand the conditions of \Cref{thm:main-result} in terms of geometrical properties of the addends $ (q_1,\ldots,q_m) \in \R[X]_{k}^m $. Throughout this section, the addends are assumed  to be generic forms. Our goal is twofold: On one hand, we want to formulate geometric criteria that are sufficient for the recovery. On the other hand, we want to justify our seemingly heuristic usage of the Gram spectrahedron. To this end, we will, for instance, examine the values of $ m $, for which uniquely representable sums of squares occur for typical choices of addends. In this section, we will prove \Cref{cor:pof-recovery-intro-gram}, \Cref{cor:pof-recovery-intro-real-variety} and \Cref{cor:pof-recovery-intro-real-variety-m=n-1}. Note that \Cref{cor:pof-recovery-intro-gram} is actually a direct consequence of \Cref{thm:main-result}: 

\begin{proof}[Proof of \Cref{cor:pof-recovery-intro-gram}] 
	If $ f_2 = \sum_{i = 1}^{m} \lambda_i q_i^2 $ is uniquely Sum-of-Squares representable, then the space $ U := \sosupp f_2 $ equals $ \langle q_1,\ldots,q_m \rangle $ by \Cref{prop:gram-spec}. Since $ q_1,\ldots,q_m $ do not satisfy any algebraic relations of degree $ 3 $, the space $ \R[U]_{3} $ has dimension $ \binom{m+2}{3} $ and \Cref{thm:main-result-intro} yields the claim. 
\end{proof}

In fact, one sees that \Cref{thm:main-result} yields uniqueness of the POF decomposition $ f_d = \sum_{i = 1}^{m} \lambda_i q_i^d, d = 2,3 $ under these two simplified conditions: 
\begin{align}
	&\circ\quad \text{There are no algebraic relations of $ q_1,\ldots,q_m $ of degree at most $ 3 $. } \label{eq:condition-algrel}\\
	&\circ\quad \sosupp(\sum_{i = 1}^{m} q_i^2) = \langle q_1,\ldots,q_m \rangle \label{eq:condition-sos}
\end{align}
The positive weights $ \lambda_1,\ldots,\lambda_m $ matter for neither of these conditions, which is why we omit them throughout the section. Let us now discuss \eqref{eq:condition-algrel} and \eqref{eq:condition-sos}.

\subsubsection{First condition: Algebraic relations of general $ k $-forms}\label{sec:first-condition}
Condition \eqref{eq:condition-algrel} is the easier one: General choices of $ q_1,\ldots,q_m $ will not have any algebraic relations at all as long as $ m\leq n $. However, since we are only interested in degree-$ 3 $ relations, the maximum value $ \beta_{k}(n) $ of $ k $-forms $ q_1,\ldots,q_{\beta_{k}(n)} \in \R[X]_{k} $ such that $ \irel(q_1,\ldots,q_{\beta_{k}(n)})_{\leq 3} = \{0\} $ could potentially be much higher, with the obvious bound $ \beta_{k}(n) \leq \binom{n+k-1}{k} $ from linear relations. In \cite[Section 6.4]{Bafna_Hsieh_Kothari_Xu_2022}, the authors give a combinatorial proof that $ \beta_{2}(n) =\Theta(n^2) $. I am not aware of any other reference for this statement. 
A numerical study combined with OEIS suggests that $ \beta_{2}(n) $ is at least $ \lceil \frac{(n+2)(n+1)}{6} \rceil$, with the lower bound being obtained from the explicit instance 
\begin{align}
	q_{ijk} := (X_i + X_j + X_k)^2, \qquad (i\leq j\leq k, i+j+k \equiv_{n} 0)
\end{align}
Note that this explicit instance would give a lower bound for generic rank-$ r $ quadratics $ q_1,\ldots,q_m $ of all ranks $ r \in \{1,\ldots,n\} $, by choosing $ \mathcal{D} := \mathcal{D}_{r} $ as the class of quadratic forms of rank $ r $ in the subsequent \Cref{prop:algindep-sequence}. 

\begin{prop}\label{prop:algindep-sequence}
	Let $ K \in \{\R, \C\} $ and $ m, n, k, d\in \N_0 $. Let $ \mathcal{D}\subseteq K[X]_{k} $ an irreducible variety containing $ m $ distinct forms that do not satisfy any relations of degree $ kd $. Let $ q_1,\ldots, q_m $ general in $ \mathcal{D} $. Then also $ \irel(q_1,\ldots,q_m)_{kd} = \{0\} $.
\end{prop}
\begin{proof} 
	Let $ d\in \N $ and $ N := \{\alpha\in \N_0^m\mid |\alpha|=d\} $. We show that there are no algebraic relations in degree $ kd $ over $ K = \C $, which clearly also shows the claim over $ K = \R $. Consider the variety
	\begin{align}
		W_{m, d} = \{[\lambda: q] \in \P(\C^{N} \times S^k(\C^n)^m) \mid \sum_{\alpha\in \N_0^m, |\alpha| = d} \lambda_{\alpha}q^{\alpha} = 0 \}  
	\end{align}
	consisting of pairs of $ (q_1,\ldots, q_m) $ and the coefficients $ \lambda $ of relations in between them. Let $ \pi $ denote  the projection to the $ q $-coordinates and consider any tuple of forms $ q = (q_1,\ldots, q_m) $. Then the fiber $ \pi^{-1}(\{q\}) $ is a (projective)  subspace of $ \P(\C^{N}) $ corresponding to the algebraic relations of $ q_1,\ldots, q_m $ in degree $ d $. The set of points $ q $ such that $ \pi^{-1}(\{q\}) $ is nonempty is Zariski closed. Thus if we find a specific point $ q $ such that  $ \pi^{-1}(\{q\}) $ is empty, it will be empty on a Zariski open neighbourhood of $ q $. But such a specific sequence $ q=(q_1,\ldots,q_m)  \in \mathcal{D} $ exists by assumption.
\end{proof}


\begin{conjecture}\label{conj:algindep-rank1-quadratics} The family 
	\begin{align}
		q_{ijk} := (X_i + X_j + X_k)^2, \qquad (i\leq j\leq k, i+j+k \equiv 0 \mod n)
	\end{align}
	of $ m = \lceil \frac{(n+2)(n+1)}{6} \rceil $ quadratics in $ n $ variables $ X = (X_1,\ldots,X_n) $ does not satisfy any algebraic relations of degree $ 3 $. 
\end{conjecture}
\begin{proof}[Evidence]
	Verified on a computer for $ n\in \{7,\ldots,12\} $. 
\end{proof}

\begin{remark}\label{rem:algindep-bezout}
	For $ m = n+1 $, $ k\ge 2 $, $ K\in \{\R, \C\} $ and general $ q_1,\ldots, q_m \in K[X]_{k} $, the ideal of relations $ I_{\mathrm{rel}}(q_1,\ldots, q_m) = (f) $ is principal and generated by some $ f \in K[Y_1,\ldots, Y_m] $ of degree $ k^n $. Indeed, 
	the polynomial map
	\begin{align}
		\underline{q}\colon K^n \to K^m, x\mapsto (q_1(x),\ldots,q_m(x)) 
	\end{align} 
	has $ n $-dimensional image. 
	The degree of $ f $ is the degree of the variety $ V(f) \subseteq \C^m $, which is determined by intersecting $ V(f) $ with a general subspace $ H $ of dimension $ 1 $. Denote by $ \mathcal{L} $ the $ n $-dimensional space of linear equations defining $ H $. Choose a basis $ \mathcal{L} = \langle \ell_{1},\ldots, \ell_{n} \rangle $ where $ \ell_{1},\ldots, \ell_{n}\in \C[Y]_{1} $. Pulling back via $ \underline{q} $ yields a system of $ n $ quadratic forms $ \ell_{i}(q_1,\ldots, q_m) \in \C[X]_{2} $, $ i\in \{1,\ldots, n\} $. By Bézout's theorem and genericity of $ H $ and $ q_1,\ldots, q_m $, we obtain that this system has $ k^n $ (complex) solutions, so $ \deg(f) = k^n $. For all $ n \ge 2, k\ge 2 $, this means that $ \beta_{k}(n)\ge n+1 $. 
\end{remark}

\subsubsection{Second condition: Unique Sum-of-Squares representations}\label{sec:second-condition}

Condition \eqref{eq:condition-sos} is more interesting. The space $ \sosupp f_2 $ is hard to analyze,  unless it equals $ \langle q_1,\ldots,q_m \rangle $. For generic $ q_1,\ldots,q_m $, their second order power sum $ f_2 = q_1^2 + \ldots + q_m^2 $ will have length $ m $ as long as $ m\leq \rk_{k}^{\circ}(n, 2k) $, cf. \Cref{def:waring-rank-and-length}. Thus, $ q_1,\ldots,q_m $ form a \textbf{minimum} length Sum-of-Squares representation of $ f_2 $. Therefore, the Gram matrix $ G(q) $ associated with this representation lies on the boundary of $ \gram(f) $, by \Cref{prop:gram-spec}. On the other hand, also by \Cref{prop:gram-spec}, the Gram matrices that have $ \sosupp f_2 $ as their image are precisely the relative interior points of $ \gram(f_2) $, which correspond to \textbf{maximum} length Sum-of-Squares representations of $ f_2 $ (with linearly independent addends). Condition \eqref{eq:condition-sos} is thus saying that the boundary point $ G(q) $ is also a relative interior point of $ \gram (f_2) $. This is only possible if $ \gram(f_2) $ is a singleton and therefore if $ f_2 $ is uniquely Sum-of-Squares representable. \Cref{prop:unique-sos-representations} collects this easy fact for future reference.

\begin{prop}\label{prop:unique-sos-representations} 
	Let $ k, m\in \N $ with $ m\leq \rk_{k}^{\circ}(n, 2k) $ and $ q_1,\ldots,q_m \in \R[X]_{k} $ be general $ k $-forms. Denote $ f_2 = \sum_{i = 1}^{m} q_i^2 $. Then $ f_2 $ is uniquely Sum-of-Squares representable, if and only if $ \sosupp f_2 = \langle q_1,\ldots,q_m \rangle $. 
\end{prop}

Now, under which geometrical conditions on the variety of $ q_1,\ldots,q_m $ does $ \sum_{i = 1}^{m} q_i^2 $ have a unique Sum-of-Squares representation? We write $ V_{\R} := V\cap \R^n $ for the set of real points of some affine variety $ V\subseteq \C^n $. 

\begin{reminder}\label{ex:dense-real-points}
	A subvariety $ V \subseteq \C^n$ has dense real points $ V_{\R}\subseteq \R^n $ if and only if every irreducible component of $ V $ contains a  real point that is smooth in $ V $. In particular, an irreducible nonsingular subvariety $ V \subseteq \C^n$ has dense real points $ V_{\R} $ if and only if it contains a real point. 
\end{reminder}

\begin{prop}[{Corollary of Bertini's theorem, cf. \cite[II 8.4(d)]{hartshorne2013algebraic}}]\label{cor:bertini} 
	Fix $ k, m\in \N $. If $ q_1,\ldots,q_m\in \C[X_1,\ldots,X_n]_{k} $ are general and $ m\leq n-1 $, then $ I:=(q_1,\ldots,q_m) $ is a radical ideal, and its variety $ V(I) $ in $\P(\C^n)$ is of pure codimension $ m $. If, in addition, $ m\leq n-2 $, then $ I $ is a prime ideal and $ V(I) $ is smooth. 
\end{prop}

\begin{lemma}\label{lem:sosupport-representation} Let $ m, k\in \N $. The following hold:
	\begin{enumerate}[(a)]
		\item Let $ q_1,\ldots,q_m\in \R[X]_{k}^m $. Assume that $ I:= (q_1,\ldots,q_m) $ is a radical ideal and $ V_{\R}(I) $ is dense in $ V(I) $. Then, for $ f_2 := \sum_{i = 1}^{m} q_i^2 $, it holds that $ \sosupp f_2 = \langle q_1,\ldots,q_m \rangle $ and $ f_2 $ has dual nondegenerate Sum-of-Squares support. 
		\item Let $ q_1,\ldots,q_m \in \R[X]_{k}^m$ be general forms satisfying the assumption of (a) and let $ m\leq \rk_{k}^{\circ}(n, 2k)$. Then $ f_2 $ is uniquely Sum-of-Squares representable.
	\end{enumerate}
\end{lemma}
\begin{proof}
	Claim (a) follows immediately from \Cref{cor:real-radical=>dual-nondegenerate}, proven in \Cref{app:pointed-sos-cones}. Everything except dual nondegeneracy can also be seen directly: 
	Let $ N\in \N_0 $ and $ p_1,\ldots,p_N \in \R[X]_{k} $ such that $ \sum_{i = 1}^{N} p_i^2 = \sum_{i = 1}^{m} q_i^2 $. Evaluating this identity in some $ x\in V_{\R}(I) $ yields that $ p_1,\ldots,p_m $ vanish on $ V_{\R}(I) $. Since $ V(I) $ has dense real points and $ I $ is radical, $ p_1,\ldots,p_m $ must lie in $ I_{k}  = \langle q_1,\ldots,q_m \rangle$. Thus $ \langle q_1,\ldots,q_m \rangle $ equals $\sosupp f_2$. Claim (b) follows from (a) using \Cref{prop:unique-sos-representations}. 
\end{proof}

\begin{corollary}[{Restatement of Corollaries \ref{cor:pof-recovery-intro-real-variety} and \ref{cor:pof-recovery-intro-real-variety-m=n-1}}]\label{cor:pof-recovery-interpretation-variety} 
	Let $ q_1,\ldots,q_m\in \R[X]_{k}^m $ general satisfying one of these properties:
	\begin{enumerate}
		\item $ m\leq n-2 $ and $ V_{\R}(q_1,\ldots,q_m) $ contains a nonzero point, or
		\item $ m = n-1 $ and all lines in the affine cone $ V(q_1,\ldots,q_m) $ are real. 
	\end{enumerate} 
	Then, $ \sum_{i = 1}^{m} q_i^2 $ is uniquely Sum-of-Squares representable, with dual nondegenerate Sum-of-Squares support. In particular, \Cref{algo:pof-recovery-general} recovers $ \{q_1,\ldots,q_m\} $ from input $ \sum_{i = 1}^{m} q_i^2 $ and $ \sum_{i = 1}^{m} q_i^3 $. 
	For fixed $ m \leq n-1 $, condition (a) or (b), respectively, are satisfied on a Euclidean open subset of $ \R[X]_{k}^m $. 
\end{corollary}
\begin{proof}
	\textbf{Case (1):} Since $ m\leq n-2 $, Bertini's Theorem \ref{cor:bertini} guarantees that $ I = (q_1,\ldots,q_m) $ is a prime ideal and the affine cone $ V(I) $ is smooth and irreducible. Thus, by \ref{ex:dense-real-points}, the condition of \Cref{lem:sosupport-representation} is satisfied. From the Poincaré-Miranda theorem \ref{thm:poincare-miranda}, it easy to see that $ V(q_1,\ldots,q_m) $ contains a real nonzero point for a Euclidean open subset of all $ k $-forms $ q_1,\ldots,q_m $. This is elaborated in \Cref{app:typical-regions-gram-spec}.

	\textbf{Case (2):} Since $ m = n-1 $, Bertini's theorem \ref{cor:bertini} yields that  $ I = (q_1,\ldots,q_m) $ is a radical ideal and $ V(I) $ is a union of finitely many lines. Thus the condition of \Cref{lem:sosupport-representation} is met if and only if all those lines are real. If the affine cone $ V(I) $ has the maximum number of $ k^m $ real lines given by the Bézout bound for some specific choice of $ q_1,\ldots,q_m $, then so it does in a neighbourhood: Indeed, if there was a curve $ q(t) $ through $ q(0) = (q_1,\ldots,q_m) $ such that $ V(q(t)) $ had nonreal lines for each $ t\ne 0 $, then a single real line of $ V(q(0)) $ would have to branch into two complex conjugate lines. This is not possible, since the specific system at $ t = 0 $ already attains the Bézout bound. For the specific choice, one may choose $ q_i\in \R[X_i, X_n]_{k} $ as $ X_n $-homogenizations of univariate polynomials in $ X_i $ with $ k $ distinct real zeros.  
\end{proof}

\subsection{Numerical experiments with trace-free quadratics}\label{sec:trace-free-quadratics-study}

\begin{figure}[h]
	\centering
	\includesvg[width = 0.5\columnwidth]{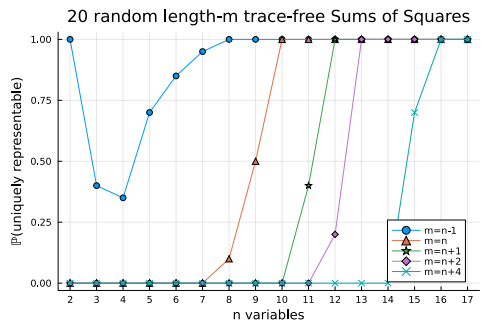}
	\caption{The $y$-axis shows the probability that $\sum_{i = 1}^{m} q_i^2 $ is uniquely representable, if $ q_1,\ldots,q_m $ are iid Gaussian random trace-free quadratics. All probabilities were empirically estimated by sampling the quadratic forms and then solving SDPs. Each data point corresponds to an average over 20 instances. The curves show the behaviour for different relations between $m$ and $n$. } 
	\label{fig:pointedness-study}. 
\end{figure}

This section conducts a numerical study on random quadratics. Instances $ q = (q_1,\ldots,q_m) $ of random quadratic forms are sampled such that $ q_1,\ldots,q_m $ are iid. Subsequently, we use semidefinite programming to determine the dimension of $ \sosupp(\sum_{i = 1}^{m} q_i^2) $ and the kernel of some relative interior point of the dual SDP from \Cref{app:pointed-sos-cones}. The quadratics are chosen from a trace-free distribution (see \Cref{def:gauss-trace-free}), to avoid choosing positive definite forms. I wish to thank Greg Blekhermann and João Gouveia for this suggestion, as it appears that for these distributions, the probability to get uniquely representable sums of squares behaves surprisingly regular. Precisely, we consider the following distributions:
\begin{defi}\label{def:gauss-trace-free}
	Let $ n\in \N $ and $ X = (X_1,\ldots,X_n) $. 
	\begin{enumerate}
		\item We call a random quadratic $ q $ in $ X $ \emph{Gaussian trace-free}, if it is sampled as follows: 
		Choose the entries of a matrix $ A\in \R^{n\times n} $ independently at random from the standard normal distribution $ \mathcal{N}(0, 1) $. Set $ q:= X^{T}(A-\tr(A)I_n)X $. 
		\item We call a random quadratic $ q $ in $ X $ \emph{Gauss-Gramian trace-free}, if it is sampled as follows: Choose the entries of a matrix $ A\in \R^{n\times n} $ independently at random from the standard normal distribution $ \mathcal{N}(0, 1) $. Then, set $ q:= X^{T}(A^{T}A-\tr(A^{T}A)I_n)X $. 
	\end{enumerate}
\end{defi}

In addition, some explicit family of $ m=\Theta(n^2) $ quadratics is constructed. We verify for the first values $ n\in \N $ that the sum of its squares is uniquely representable. We conjecture that this holds true for all $ n \in \N $, see \Cref{conj:explicit-quadratic-tracefree-family}. The code and results of all  experiments can be found on {GitHub}, see \cite{Taveira_Blomenhofer_Pofdecomp_Github_2023}. Computations were done in \texttt{Julia} \cite{Julia-2017}, with the packages \texttt{JuMP} \cite{JuMP_DunningHuchetteLubin_2017}, \texttt{MultivariatePolynomials} \cite{legat2021multivariatepolynomials}, \texttt{SumOfSquares} \cite{jl-sumofsquares-1}, \cite{jl-sumofsquares-2} and the \texttt{Mosek} solver \cite{mosek}. For the Gaussian trace-free distribution, this is the observed behaviour of the probability $ p_{m, n} $ that $ \sum_{i = 1}^{m} q_i^2 $ is uniquely representable and its supporting face in $ \Sigma_{2k} $ is exposed: 
\begin{enumerate}
	\item If $ m = n+r $ for some constant $ r $, then $ p_{m, n} $ appears to be an S-shaped curve in $ n $ that converges to $ 1 $ as $ n\to \infty $. This behaviour is depicted in \Cref{fig:pointedness-study}.  Cf.  \cite[\texttt{data/experiment-3}]{Taveira_Blomenhofer_Pofdecomp_Github_2023}. 
	\item $ p_{m, n} > 0 $ if $ m \leq n-1 $. See the blue curve in \Cref{fig:pointedness-study}. This is consistent with the results from \Cref{sec:real-geometry}. Cf.  \cite[\texttt{data/experiment-3}]{Taveira_Blomenhofer_Pofdecomp_Github_2023}. 
	\item For $ m(n) = \lceil \frac{(n+2)(n+1)}{6} \rceil = \Theta(n^2) $, it appears to hold $ p_{m(n), 2n} \approx 1 $, but $ p_{m(n), n} \approx 0 $. Cf. \cite[\texttt{data/experiment-2}]{Taveira_Blomenhofer_Pofdecomp_Github_2023}.
\end{enumerate}

The same qualitative behaviour holds true if the Gaussian trace-free distribution is replaced by the Gauss-Gramian trace-free distribution from \Cref{def:gauss-trace-free}(b). Cf. \cite[\texttt{data/experiment-2-gramian} and [\texttt{data/experiment-3-gramian}]{Taveira_Blomenhofer_Pofdecomp_Github_2023}. All of the above statements are empirical observations, based on limited computational experiments in a bounded number of variables. A modest version of the first observation is formulated in the following Conjecture.

\begin{conjecture}\label{conj:pof-recovery-intro-gram}
	If $ q_1,\ldots,q_{n} $ are chosen as iid Gaussian random trace-free quadratics, then $ \sum_{i = 1}^{n} q_i^2 $ is uniquely representable with probability $ p_{n} \to 1 $ as $ n\to \infty $. 
\end{conjecture}

The third observation aligns with \Cref{conj:explicit-quadratic-tracefree-family}, for which we gathered separate numerical evidence. In particular, both suggest that there exist open neighbourhoods of parameters $ q = (q_1,\ldots,q_m) $, where $ m $ is much larger than $ n $ and their sum of squares is uniquely representable. This leads to a natural question: What is the maximum typical length of a uniquely representable sum of squares? Formally: 

\begin{question}\label{quest:singleton-gramspec-maxrank}
	For $n\in \N$, what is the maximum number $m(n)$ of linearly independent quadratics $q_1,\ldots,q_{m(n)}$ in $n$ variables such that for all $ p_1,\ldots,p_{m(n)} $ in some (Euclidean) neighbourhood of $q_1,\ldots,q_{m(n)}$, $ \sum_{i = 1}^{m(n)} p_i^2$ is uniquely representable?
\end{question}
\begin{conjecture}\label{conj:explicit-quadratic-tracefree-family}
	The explicit family of $ m = \lceil \frac{(n+2)(n+1)}{6} \rceil$ trace-free quadratics 
	\begin{align}
		q_{ijk} := (X_i + X_j + X_k)(Y_i + Y_j + Y_k), \qquad (i\leq j\leq k, i+j+k \equiv_{n} 0)
	\end{align}
	in $ 2n $ variables $ (X, Y) = (X_1,\ldots,X_n,Y_1\ldots,Y_n) $ does not satisfy any algebraic relations of degree $ 3 $ and 
	\begin{align}
		\sum_{\substack{i\leq j\leq k\\ i+j+k \equiv_{n} 0}} q_{ijk}^2 
	\end{align} 
	is uniquely Sum-of-Squares representable, with nondegenerate dual. The same claim also holds true if $ q_{ijk} $ are replaced by $ q_{ijk} + \varepsilon p_{ijk} $, where $ \varepsilon\in \R $ is sufficiently small and $ p_{ijk} $ is some polynomial with support contained in $ \{ X_iY_j \mid i, j = 1,\ldots,n \} $.
\end{conjecture}
\begin{proof}[Evidence]
	In \cite[\texttt{data/experiment-1}]{Taveira_Blomenhofer_Pofdecomp_Github_2023}, we checked numerically for $ n=2,\ldots,15 $. Note that  \Cref{conj:algindep-rank1-quadratics} would imply that also the specific family here has no algebraic relations up to degree $ 3 $, since this one can be reduced to the one from \Cref{conj:algindep-rank1-quadratics} by substituting $ Y_i\mapsto X_i $ for all $ i\in \{1,\ldots,n\} $. 
\end{proof}

%% file: applications.tex
\noindent
Let us now discuss two practical applications of powers-of-forms decomposition. 

\subsection{Mixtures of centered Gaussians}\label{sec:gmm}

\begin{figure}[h]
	\centering
	\includesvg[width = \columnwidth]{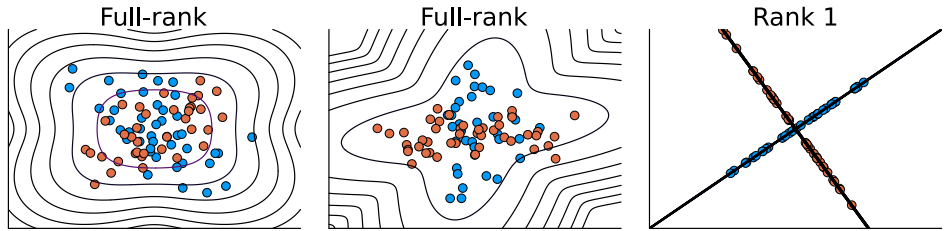}
	\caption{
		Various centered Gaussian mixture distributions of rank $ 2 $. The sample colouring indicates which of the Gaussians was chosen in the sampling process. The pictures on the left and middle show mixtures of full-dimensional Gaussians, as treated in \Cref{sec:gmm}. Here, the contour lines describe the probability density function of the mixture. The rightmost picture shows a mixture on proper subspaces. These are addressed in \Cref{sec:learn-union-subspaces}.}
	\label{fig:gmm-density-clustering}
\end{figure}

The distribution $ \mathcal{N}(\mu, \Sigma) $ of a Gaussian random vector on $ \R^n $ is parameterized by its mean vector $ \mu \in \R^n$ and its symmetric positive definite covariance matrix $ \Sigma \in \R^{n\times n} $. A Gaussian random vector is called \emph{centered}, if its mean vector is zero. In that case, all information about its distribution is contained in the quadratic form $ q := X^{T}\Sigma X $. A \emph{Gaussian mixture} $ Y $ is a random variable that is sampled as follows: From a box containing $ m $ normally distributed random variables $ Y_1,\ldots,Y_m $, blindly draw one of them (with $ \lambda_i \ge 0 $ being the probability to draw $ Y_i $, assuming $ \sum_{i = 1}^{m} \lambda_i = 1 $), and then sample $ Y_i $.  We denote $ Y = \lambda_1 Y_1 \oplus \ldots \oplus \lambda_m Y_m $ for the random variable $ Y $ defined by this sampling procedure. Let us now consider a mixture $ Y = \lambda_1\mathcal{N}(0, \Sigma_1)\oplus \ldots \oplus \lambda_m \mathcal{N}(0, \Sigma_m) $ of centered Gaussians with covariance forms $ q_i = X^{T}\Sigma_i X $. It turns out that from sufficiently many samples of $ Y $, it is possible to compute (noisy versions of) the expressions
\begin{align}\label{eq:gmm-moments}
	\sum_{i = 1}^{m} \lambda_i q_i^{d},\qquad d = 0,\ldots,D
\end{align}
where $ D\in \N $ is some threshold depending on the order of the number of samples. Up to scalars, the expressions \eqref{eq:gmm-moments} are the degree-$ 2d $ \emph{moment forms} of $ Y $, i.e. the $ 2d $-homogeneous parts of the \emph{moment generating series} $ \E_{\sim Y}[\exp(Y^{T}X)] $ of $ Y $. The connection between samples, moments and powers-of-forms expressions is explained with lots of details in my doctoral thesis \cite[Chapter 3, Introduction]{Taveira_Blomenhofer_Thesis}.

The goal is to estimate the parameters $ q_1,\ldots,q_m $ from not too many samples. The \emph{moment problem for mixtures of Gaussians} asks to recover the parameters from (exact) knowledge of the moment forms instead. It can be seen as a coarsening of the statistical estimation problem: To estimate the expression $ \sum_{i = 1}^{m} \lambda_i q_i^d $, one needs roughly $ \mathcal{O}(\sigma^{d}) $ iid samples, where the noise from estimation can be made arbitrarily small by taking more samples, $ \sigma $ is a total variance parameter depending on $ q_1,\ldots,q_m $ and the $ \mathcal{O} $-Notation hides e.g. dependency on the dimension $ n $. 
In order to be efficient with the sample complexity, it is therefore desirable to keep $ D $ as small as possible. Previous work of the author with Casarotti, Micha{\l}ek and Oneto \cite{Blomenhofer_Casarotti_Michalek_Oneto_2022} showed that theoretical identifiability of the parameters holds true in our setting, if $ D \ge 3 $ and $ m $ is not too large (roughly $ m\in \mathcal{O}(n^{d-1}) $). This explains the focus of the present paper on the minimal case $ D = 3 $. That being said, we are now ready to give a proof of \Cref{cor:cgmm-recovery-intro}:

\begin{proof}[Proof of \Cref{cor:cgmm-recovery-intro}] 
	Writing $ q_i = X^{T}\Sigma_iX $, the moment forms $ \mathcal{M}_{2d}(Y) $ of the Gaussian mixture random variable $ Y = \lambda_1\mathcal{N}(0, \Sigma_1) \oplus \ldots \oplus\lambda_m \mathcal{N}(0, \Sigma_m) $ may be expressed as a convex combination (cf. \cite[Chapter 3, Introduction]{Taveira_Blomenhofer_Thesis})
	\begin{align}
		\mathcal{M}_{2d}(Y) = c_d \sum_{i = 1}^{m} \lambda_i q_i^d, \qquad (c_d\in \Q),
	\end{align} 
	and these are given as input for $ d\in \{0,\ldots,3\} $. The combinatorial expression $ c_d $ is explicitly known, see \cite[Chapter 3, Introduction]{Taveira_Blomenhofer_Thesis}. By \Cref{cor:pof-recovery-intro-real-variety}, we know that there exists a Euclidean open subset $ \mathcal{U} $ of $ m $-tuples of quadratics where \Cref{algo:pof-recovery-general} recovers the quadratics from their third and second order powers sums. Now, fix some positive definite form $ p $ and observe that for $ \lambda\in \R_{>0} $ sufficiently large, $ \mathcal{U}' := \{q\in \mathcal{U} + \lambda p \mid q_1 \succ 0,\ldots,q_m\succ 0 \} $ will be a nonempty Euclidean open subset of tuples of positive definite quadratics. On this subset, the following algorithm works: 
	
	\noindent
	Choose a new variable $ Z $ and compute 
 	\begin{align}\label{eq:power-sum-shift}
		f_2 = \sum_{i = 1}^{m} \lambda_i (q_i - Z)^2, \quad f_3 = \sum_{i = 1}^{m} \lambda_i (q_i - Z)^3.
	\end{align}
	from the input. This is possible, since the $ X $-homogeneous parts of the power sums from \eqref{eq:power-sum-shift} correspond to the moments forms $ \mathcal{M}_0(Y),\mathcal{M}_{2}(Y),\ldots,\mathcal{M}_6(Y) $. Plug in $ Z\mapsto \lambda p $ to obtain an instance where \Cref{algo:pof-recovery-general} succeeds to recover the addends $ q_1-\lambda p,\ldots,q_m-\lambda p $, with corresponding weights. Shift back by $ \lambda p $ to recover the covariance forms.   
\end{proof}

\begin{remark}\label{rem:gauss-gramian}
	The numerical experiments, which lead to \Cref{conj:pof-recovery-intro-gram}, suggest that the shifting method from the proof of \Cref{cor:cgmm-recovery-intro} works with asymptotic probability $ 1 $ in an ``average case'' framework, where the covariance forms $ q_1,\ldots,q_m $ are constructed from ``random positive definite matrices'':
	Consider $ q_1,\ldots,q_m $ sampled iid from the distribution $ \frac{1}{n^2} X^{T}A^{T}AX $, where $ A $ is a random $ n\times n $ matrix with iid Gaussian  distributed entries in $ \mathcal{N}(0, \sigma^2) $. Here, $ \frac{1}{n^2}\tr(A^{T}A) $, which is the squared Frobenius norm of $ \frac{1}{n}A $, will concentrate around the expected value, which is $ \sigma^2 $, with high probability. Thus for large $ n $, $ q_1,\ldots,q_m $ will all have roughly the same trace. A proxy for this value is algorithmically accessible, since $ \sigma^2 \approx \frac{1}{m} \tr(\sum_{i = 1}^{m} q_i) $. 
	A posteriori, this motivates \Cref{def:gauss-trace-free}(b) of Gauss-Gramian quadratic forms, since that will be the distribution of Gaussian random psd forms after shifting by the trace. 
\end{remark}

\subsection{Learning a union of subspaces}\label{sec:learn-union-subspaces}

Learning unions of subspaces is a comparatively new problem that emerged from applications in computer vision and dimensionality reduction techniques in data science (\cite{Wang_Wipf_Ling_Chen_Wassell_2015},\cite{Hegde_Indyk_Schmidt_2016a},\cite{Lipor_Balzano_2017},\cite{Hong_Malinas_Fessler_Balzano_2018}). 
It assumes that the given data stems from a distribution, whose support is a union of $ r $-dimensional subspaces $ U_1,\ldots,U_m $. $ r $ is known and the objective is to find bases for the subspaces $ U_1,\ldots,U_m $ from samples of the mixture distribution as input. A union of subspaces is an algebraic variety. E.g., if the subspaces are hyperplanes defined by linear forms, then their union is the zero set of the product of the linear forms. Note that without a distributional assumption, recovering this variety is likely the best one can do, and it can be a hassle to recover the high-degree polynomials describing it from data. 

However, assuming that the data on the individual subspaces is Gaussian distributed, it is often possible to recover the subspaces using degree-$ 6 $ moments of the empirical data. More so, it is then possible to learn the distributions on the subspaces. Indeed, note that this Gaussian subspace learning is a generalized version of the Gaussian mixture problem from \Cref{sec:gmm}, with the difference that Gaussians on subspaces will lead to psd forms $ q_1,\ldots,q_m $ that are not of full rank. 

\begin{corollary}[{Restatement of \Cref{cor:unionspaces-recovery-intro}}]\label{cor:unionspaces-recovery-restatement}
	For any $ n, r\in \N $, $m \leq n-1$, there is a Euclidean open subset $ \mathcal{U} $ of the problem parameters and an efficient algorithm for the following problem: If $ Y_1,\ldots,Y_m $ are normally distributed random variables on $ r $-dimensional subspaces $ U_1,\ldots,U_m $ and $ \lambda\in \R_{> 0}^m $ with $ \sum_{i = 1}^{m} \lambda_i = 1 $, compute bases for the subspaces $ U_1,\ldots,U_m $ from the moments of $ \lambda_1 Y_1 \oplus \ldots \oplus \lambda_m Y_m $ of degree at most $ 6 $.
\end{corollary}
\begin{proof} 
	For each $ i\in \{1,\ldots,m\} $, there exists a unique quadratic form $ q_i $ on $ \R^n $ such that the restriction of $ q_i $ to $ U_i $ equals the covariance form of $ Y_i $ and the kernel of $ q_i $ is the orthogonal complement $ U_i^{\perp} $ of $ U_i $ (with respect to the standard inner product). It is then not too hard to see that the even degree moment forms of degree at most $ 6 $ of $ \lambda_1 Y_1 \oplus \ldots \oplus \lambda_m Y_m $, up to known scalars, attain the form 
	\begin{align}\label{eq:subspace-recovery-1}
		\sum_{i = 1}^{m} \lambda_i q_i^d, \qquad d=0,1,2,3.
	\end{align}
	Write $ \mathcal{D}_r $ for the class of quadratic forms of rank at most $ r $. If $ r = 1 $, then the $ q_i $ are squares of linear forms and we can directly use an algorithm for $ 1 $-Waring decomposition, similar to the one from \Cref{app:polinearforms}. Thus, wlog $ r\ge 2 $. 
	
	Consider first a special instance: For $ m\le n-1 $, $ I = (X_1^2-X_n^2,\ldots,X_{m}^2-X_n^2) $ is a radical ideal: Indeed, if $ \mathfrak{P} $ is a minimal prime containing $ I $, then by Krull's Hauptidealsatz, $ \mathfrak{P} $ has length at most $ m $. In addition, for each $ i\in \{1,\ldots,m\} $, we have $ X_1-X_n\in \mathfrak{P} $ or $ X_1 + X_n \in \mathfrak{P} $. Thus, there exists a sign choice $ \sigma\in \{\pm 1\}^{m} $ such that $ (X_1 + \sigma_1 X_n,\ldots, X_m + \sigma_m X_n) \subseteq \mathfrak{P} $. Since the ideal to the left is prime and of height $ m $, we have equality. We conclude that any minimal prime containing $ I $ is of the form $\mathfrak{P} = (X_1 + \sigma_1 X_n,\ldots, X_m + \sigma_m X_n) $. A quick exercise shows that the intersection of those $ 2^{m} $ primes is indeed $ I $. Thus, $ I $ is radical. Denote $ q_{\mathrm{spec}} :=  (X_i^2-X_n^2)_{i=1,\ldots,m} $ for those special quadratic forms. 
	
	By the Kleiman-Bertini theorem \cite[Appendix B, 9.2]{fulton2012intersection}, the set of rank-$ r $ forms $ q = (q_1,\ldots,q_m) $ such that $ V(q_1,\ldots,q_m) $ is not of pure codimension $ m $, is closed. Thus, we find a Zariski open neighbourhood $ \mathcal{U} $ of $ q_{\mathrm{spec}} $ such that for all $ q\in \mathcal{U} $, all irreducible components of $ V(q) $ have codimension $ m $. 
	Consider thus a general subspace $ \mathcal{H} $ in $ \C^n $ of fitting dimension such that the intersection $ \mathcal{H}\cap V(q) $ consists of $ \deg V(q) $ many lines (or, projective points). Since $ \deg V(q_{\mathrm{spec}}) $ attains the Bézout bound, by semicontinuity of the degree, all $ q $ in a neighbourhood will satisfy $ \deg V(q) = 2^m $, too. For $ q = q_{\mathrm{spec}} $, all $ 2^m $ projective points in $ \mathcal{H}\cap V(q) $ are real. This property must also hold in a Euclidean neighbourhood: Indeed, if a sequence of general $ q^{(n)} \in \R[X]_{2}^m $ existed such that $ q^{(n)} \to q_{\mathrm{spec}} $ as $ n\to \infty $ and $ H\cap V(q^{(n)})$ did contain non-real points, then a pair of complex conjugate non-real solutions $ (z_n, z_n^{\ast}) $ would degenerate to just one real solution. Since $ H\cap V(q)$ has a constant number of $ 2^m $ points in a neighbourhood of $ q_{\mathrm{spec}} $, this is not possible. As $ \mathcal{H} $ was an arbitrary (general) hyperplane, it follows that $ V(q) $ has dense real points in a Euclidean neighbourhood $ \mathcal{U}' $ of $ q = q_{\mathrm{spec}} $. 
	
	By \Cref{lem:sosupport-representation}(b), \Cref{algo:pof-recovery-general} succeeds to recover the addends, if the (weighted) power sums are constructed from elements of $ \mathcal{U}' $. Therefore, on the open subset $ \mathcal{D}_r \cap (\mathcal{U} + 2X_n^2) $ of quadratics of rank $ r $, the following algorithm works: 
	
	\begin{enumerate}
		\item Choose a new variable $ Z $ and compute 
		\begin{align}\label{eq:power-sum-shift-2}
			\sum_{i = 1}^{m} (q_i - Z)^2\quad \text{ and } \quad \sum_{i = 1}^{m} (q_i - Z)^3.
		\end{align}
		from the input.
		\item Plug $ Z\mapsto 2X_n$ into the forms from \eqref{eq:power-sum-shift} to obtain power sums $ f_2, f_3 $, whose unique decomposition has addends in $ \mathcal{U} $. 
		\item Use \Cref{algo:pof-recovery-general} on input $ f_2, f_3 $ to compute some $ (\hat{q}_1,\lambda_1),\ldots,(\hat{q}_m,\lambda_m)$. 
		\item Output $ \{(\hat{q}_1+2X_n^2,\lambda_1),\ldots,(\hat{q}_m+2X_n^2,\lambda_m)\}$.
	\end{enumerate}
	
	\medskip
	The set $ \mathcal{D}_r \cap \mathcal{U}' $ is open in $ \mathcal{D}_r $ and intersects the subset $ \psd_r $ of rank-$ r $ psd quadratics in the point  $ q = (X_i^2-X_n^2)_{i=1,\ldots,n-1} $. Every neighbourhood of a point in $ \psd_r $ contains points in the interior of $ \psd_r $. This means that $ \mathcal{U}' \cap \psd_r $ contains a nonempty open subset $ \mathcal{U}'' $ of $ \psd_r $, showing the claim. Note that the weights can be arbitrary positive reals, summing up to $ 1 $.
\end{proof}

%% file: appendix.tex

\section{Powers of Linear forms}\label{app:polinearforms}

\begin{theorem}[{Restatement of \Cref{thm:undercomplete-waringdecomp}}]\label{thm:undercomplete-waringdecomp-restatement}  
	There exists an efficient algorithm that, on input $ m, n \in \N$ and forms $ f_2 $, $ f_3 $ of degrees $ 2 $ and $ 3 $, respectively, computes the solution to the following problem: If $ f_2, f_3 $ have a POF decomposition
	\begin{align}\label{eq:undercomplete-waringdecomp-restatement-1}
		f_d = \sum_{i = 1}^{m} \lambda_i \ell_i^d
	\end{align} 
	such that $ \ell_{1},\ldots,\ell_{m} $ are linearly independent and $ \lambda_1,\ldots,\lambda_m\in \R\setminus\{0\} $, then compute $ {(\ell_{1},\lambda_1),\ldots,(\ell_{m},\lambda_m)} $. Under these conditions, \eqref{eq:undercomplete-waringdecomp-restatement-1} is the unique minimum rank POF decomposition of $ (f_2, f_3) $ and the only POF decomposition with linearly independent addends. 
\end{theorem}
\begin{proof}[Algorithmic Proof] 
	First, note that a partial derivative of $ f_2 $ in direction $ v\in \R^n $ has the form $ \partial_{v} f_2 = 2\sum_{i = 1}^{m} \lambda_i \ell_{i}(v) \ell_{i} $.   
	The set $ \{\partial_{v} f_2 \mid v\in \R^n \} $ equals $ \langle \ell_{1},\ldots,\ell_{m} \rangle $, by linear independence. Thus, we may compute some basis $ u = (u_1,\ldots,u_m) $ of $ U:= \langle \ell_{1},\ldots,\ell_{m} \rangle $. Then, there exist vectors $ a_1,\ldots,a_m\in \R^m $ such that $ \ell_{i} = a_i^{T}u $. 
	Compute now the partial derivative 
	\begin{align}
		f_v := \frac{1}{3}\partial_{v} f_3 = \sum_{i = 1}^{m} \lambda_i (a_i^{T}u(v)) (a_i^{T}u)^2 = u^{T}M_v u
	\end{align}
	of $ f_3 $ in some random direction $ v\in \R^n $. Here, we write $ M_v := \sum_{i = 1}^{m} \lambda_i (a_i^{T}u(v)) a_ia_i^{T} $. Similarly, the quadratic form $ f_2 $ can be written as $ f_2 = u^{T}Mu $, where the matrix $ M = \sum_{i = 1}^{m} \lambda_i a_ia_i^{T} \in \R^{m\times m} $ is symmetric and psd. Note that the matrices $ M $ and $ M_v $ can easily be computed from $ u $, $ f_2 $ and $ f_v $. The claim is now that the generalized eigenvalue problem 
	\begin{align}
		\det(M_v-\mu M) = 0, \qquad \mu\in \R
	\end{align}
	has $ m $ onedimensional eigenspaces, with corresponding eigenvalues $ \mu_i := a_i^{T}u(v) $. Indeed, since $ v $ was chosen at random from some continuous distribution, $ M_v $ is of full rank $ m $, with probability $ 1 $. The rank of 
	\begin{align}
		M_v-\mu M = \sum_{i = 1}^{m} \lambda_i ((a_i^{T}u(v))-\mu) a_ia_i^{T}
	\end{align} 
	drops to $ m-1 $ precisely if $ \mu = \mu_j := a_j^{T}u(v) $ for some $ j\in \{1,\ldots,m\} $. Hence, these are the eigenvalues. By randomness of $ v $, the eigenvalues are pairwise distinct. Choose eigenvectors $ x_1,\ldots,x_m $, satisfying the generalized eigenvalue equation
	\begin{align}
		M_vx_j = \mu_j Mx_j 
	\end{align}
	Writing this equation out and comparing coefficients with respect to the basis $ a_1,\ldots,a_m $, we get that $ (a_i^{T}u(v)) (a_i^{T}x_j) = \mu_j (a_i^{T}x_j) $ and therefore $ a_i^{T}x_j = a_{j}^{T}x_j \cdot \delta_{ij} $ for each $ i, j\in \{1,\ldots,m\} $. Therefore, with $ b_j := Mx_j = \lambda_j (a_{j}^{T}x_j) a_j $, we recovered a multiple of $ a_j $. It remains to recover the missing multiples and the weights. To this end, note that the values $ \mu_j = a_{j}^{T}u(v) $ and $ b_j^{T}u(v) $ are known to us, so we can compute $ a_j = \frac{\mu_j}{b_j^{T}u(v)} b_j $ and thus $ \ell_{i} = a_i^{T}u $. For the weights, solve the linear system
	\begin{align}
		f_2 = \sum_{i = 1}^{m} \nu_i \ell_{i}^2, \qquad \nu_1,\ldots,\nu_m\in \R.
	\end{align}
	Since the $ \ell_{i}^{2} $ are linearly independent, this system will have the unique solution $ \nu_j = \lambda_j$. This concludes the algorithmic part of the proof. 
	
	Regarding the uniqueness statement: For any other decomposition $ f_d = \sum_{i = 1}^{m'} \rho_i l_i^d $ with $ m'\in \N $, $ d\in \{2,3\} $, linear forms $ l_i\in \R[X] $ and $ \rho_i\in \R\setminus \{0\} $, one easily sees that the space $ U = \{\partial_{v} f_2 \mid v\in \R^n \} $ must be contained in $ \langle l_1,\ldots,l_{m'} \rangle$. Since $ \dim U = m $, this means that $ m' \ge m $, with equality if and only if $ l_1,\ldots,l_m $ are linearly independent. If $ m = m' $, then the two decompositions are therefore both equal to the output of the algorithm. This means they must be equal. 
\end{proof}

\section{Sum-of-Squares Pointedness and strict complementarity}\label{app:gram-structure}

Sum-of-Squares representations of some form $ f\in \R[X]_{2k} $ may be found algorithmically via the following primal-dual pair of semidefinite programs:
\begin{align}\label{eq:gram-primal}
	(\gram)\quad \qquad \find   \qquad &G \succeq 0\\
				   \subjto \qquad &[X]_{k}^{T} G [X]_{k} = f\nonumber
\end{align}
\begin{align}\label{eq:gram-dual}
	(\gram)^{\ast} \qquad \minimize  \qquad &E(f) \\
	\subjto \qquad &E\in \R[X]_{2k}^{\dual} \nonumber \\
	 &M_E \succeq 0\nonumber
\end{align}
Here, for a functional $ E\in \R[X]_{2k}^{\dual} $ we denote by $ M_E := (E(X^{\alpha+\beta}))_{|\alpha|=k=|\beta|} $ the so-called \emph{moment matrix} of $ E $. The moment matrix encodes a¸ psd bilinear form $ (p, q)\mapsto E(pq) $ associated with $ E $. It is well-known that $ E \in \Sigma_{2k}^{\ast} $ if and only if $ M_E $ is psd. Note that the dual problem has $ 0 $ as its optimal value. The set of optimal solutions defines a face 
\begin{align}
	C_f = \{E\in \Sigma_{2k}^{\ast} \mid E(f) = 0 \}
\end{align}
of the dual cone $ \Sigma_{2k}^{\ast} $. By complementarity, for each optimal pair $ (G, E) $ it holds that $ M_E \cdot G = 0 $. In other words,  $ \im G \subseteq \ker M_E $. Taking $ G \in \relint \gram(f)$, we see that each $ E\in C_f $ satisfies  $ \sosupp f \subseteq \ker M_E $ by \Cref{prop:gram-spec}. The latter can also be seen directly: If $ \lambda > 0 $ such that $ f-\lambda p^2 \in \Sigma_{2k} $, then  $ E(p^2) = 0 $ for all $ E\in C_f $ and thus by the Cauchy-Schwarz inequality applied to the psd bilinear form $ M_E $, $ E(ph)^2 \leq E(p^2)E(h^2) = 0 $ for each $ h\in \R[X]_{k} $, implying $ p\in \ker M_E $. It is easy to see that the space $ \ker M_E $ is constant among all $ E\in \relint C_f $. 

\emph{Strict complementarity} is the property $ \im G = \ker M_E $. If it holds for some pair $ (G, E) $ of primal-dual optimal solutions, then it has to hold for all $ G\in \relint \gram (f) $ and $ E\in \relint C_f $. In that case, we have $ \ker M_E = \sosupp f $ for all $ E\in \relint C_f $ and we say that $ f $ has \emph{dual nondegenerate Sum-of-Squares support}. Dual nondegeneracy can be useful from a practical perspective, since \eqref{eq:gram-dual} has nice properties, e.g., a full-dimensional feasible region. Geometrically, it means that the supporting face of $ f $ in $ \Sigma_{2k} $ is exposed (by any $ E\in \relint C_f $). 

\subsubsection{Relation to pointed Sum-of-Squares cones}\label{app:pointed-sos-cones} 
In this section, we will see that in the ``geometric'' settings of \Cref{cor:pof-recovery-intro-real-variety} and \Cref{cor:pof-recovery-intro-real-variety-m=n-1}, the second order power sum has dual nondegenerate Sum-of-Squares support. The argument relies on work of Blekherman, Smith and Velasco \cite{Blekherman_Smith_Velasco_2019} and examines pointedness of Sum-of-Squares cones in quotient algebras of the polynomial ring.

\begin{prop}\label{prop:blekherman-sos-pointed} (cf. Prop. 2.5. in \cite{Blekherman_Smith_Velasco_2019}). 
	Let $ k, n\in \N $ and $ I\subseteq \R[X] $ a homogeneous ideal with graded coordinate ring $ R = \R[X]/I $ such that
	\begin{align}\label{eq:pointedsos-restricted-radicality}
		\forall p\in R_{k}\colon p^2 \in I_{2k} \implies p\in I_{k}
	\end{align}
	Then the following are equivalent:
	\begin{enumerate}[(a)]
		\item The cone $ \Sigma_{R, 2k} $ is \emph{pointed}, i.e. it is closed and contains no lines. 
		\item No nontrivial sum of squares of forms of degree $ k $ equals zero in $ R_{2k} $.
	\end{enumerate}
\end{prop}
\begin{proof}
	\textbf{(a)$ \implies $(b):} By contraposition. Let there be $ N\in \N $ and $ p_1,\ldots,p_N \in R_{k} $ such that $ \sum_{i = 1}^{N} p_i^2 = 0 $. Since no nontrivial squares lie in $ I_{2k} $ by assumption, it holds that $ N\ge 2 $. Thus $ p_1^2 = -\sum_{i = 2}^{N} p_i^2 $ lies both in $ \Sigma_{R, 2k} $ and $ -\Sigma_{R, 2k} $. It follows that $ \Sigma_{R, 2k} $ contains a line.\\
	\textbf{(b)$ \implies $(a):} By assumption, $ \Sigma_{2k} $ cannot contain a line, since otherwise there would be some nonzero $ f \in \Sigma_{2k}\cap -\Sigma_{2k}$ and the nontrivial Sum-of-Squares $ f + (-f) $ would be zero. It remains to show that $ \Sigma_{2k} $ is closed. Fix a norm on the real vector space $ R_{k} $ and denote
	\begin{align}
		K := \{p^2 \in R_{2k} \mid \|p\| = 1\}
	\end{align}
	Then $ K $ is a compact basis of the cone $ \Sigma_{2k} $: 
	Indeed, $ K $ is compact since it is the image of a compact set under a continuous function. 
	It holds $ 0\notin \conv K $ by assumption, since no nontrivial sum of squares is zero in $ R_{2k} $. Therefore the cone generated by $ K $, which is $ \Sigma_{2k} $, is closed. 
\end{proof}

\begin{prop}\label{prop:dual-nondegeneracy}
	Let $ f\in \Sigma_{2k} $. Write $ I := (\sosupp f) $ for the homogeneous ideal generated by the Sum-of-Squares support and $ R := \R[X]/I $ for the quotient algebra graded by the degree. Consider the statements:
	\begin{enumerate}[(i)]
		\item $ f $ has dual nondegenerate Sum-of-Squares support.
		\item $ C_f $ spans $ I_{2k}^{\perp} $. 
		\item The cone $ \Sigma_{R, 2k}^{\ast} \subseteq R_{2k}^{\dual} $ is full dimensional.  
		\item The cone $ \Sigma_{R, 2k} \subseteq R_{2k} $ is pointed, i.e. it is closed and contains no lines.
		\item $\forall p_1,\ldots,p_N\in \R[X]_{k}\colon p_1^2 + \ldots +  p_N^2 \in I_{2k} \implies p_1,\ldots,p_N\in I_{k}$. 
	\end{enumerate}
	Then it holds $ (v)\implies (iv) \implies (iii) \Longleftrightarrow (ii) \Longleftrightarrow (i) $. 
\end{prop}
\begin{proof}
	``(v)$ \implies $(iv)'': Note that in \Cref{prop:blekherman-sos-pointed}, (b)$ \implies $ (a) also holds without the assumption \eqref{eq:pointedsos-restricted-radicality}. 
	
	\noindent
	``(iv)$\implies$ (iii)'': First, assume $ \Sigma_{R, 2k} $ is pointed in $ R_{2k} $. If $  \Sigma_{R, 2k}^{\ast} \subseteq H $ was contained in a hyperplane $ H $, then by standard properties of the dual, $ \Sigma_{R, 2k}^{\ast\ast}  $ would contain the line $  H^{\perp} $. Since $ \Sigma_{R, 2k} $ is closed, this would yield the contradiction $ H^{\perp} \subseteq \Sigma_{R, 2k}^{\ast\ast} = \overline{\Sigma_{R, 2k}} = \Sigma_{R, 2k} $. 
	
	\noindent
	``(iii) $\Longleftrightarrow$ (ii)'': The quotient map $ \pi\colon\R[X]_{2k}\to R_{2k} $ yields a pullback $ \pi^{\ast}\colon R_{2k}^{\dual} \to \R[X]_{2k}^{\dual}, L\mapsto L\circ \pi $. The map $ \pi^{\ast} $ is a bijection onto its image 
	\[ 
		 I_{2k}^{\perp} = \{L\in \R[X]_{2k}^{\dual} \mid L \text{ vanishes on } I_{2k}\}.
	\]
	The image of $ \Sigma_{R, 2k}^{\ast} $ under $ \pi^{\ast} $ equals $ \Sigma_{2k}^{\ast} \cap I_{2k}^{\perp} $, which is precisely $ C_f $. 
	
	\noindent
	``(ii)$\implies$ (i)'': 
	Assume $ C_f $ spans $ I_{2k}^{\perp} $. One easily sees that $ I_{2k}^{\perp} $ consists precisely of those functionals $ L $ such that $ I$ is contained in the kernel of the moment matrix $ M_L\colon (p, q) \mapsto L(p, q) $ of $ L $. Equality $ I=\ker M_L $ holds for general elements $ L $ of $ I_{2k}^{\perp} $ and thus also for relative interior points of $ C_f $. Thus $ f $ has dual nondegenerate Sum-of-Squares support.
	
	\noindent
	``(i)$\implies$ (ii)'': Assume $ f $ has dual nondegenerate Sum-of-Squares support. Then any $ E\in \relint C_f $ satisfies $ \ker M_E = \sosupp f = I_{k}$. 
	Now, let $ L\in I_{2k}^{\perp} $. Since $ \ker M_L \supseteq \sosupp f $, it is easy to see that for any psd matrix $ B $ with $ \ker B = \sosupp f $, it holds that $ M_L + \lambda B $ is psd for all sufficiently large $ \lambda\in \R_{>0} $. Since $ f $ has dual nondegenerate Sum-of-Squares support, we may choose some $ E\in \relint C_f $ and take $ B = M_E $ and $ \lambda\in  \R_{>0} $ such that $ M_L + \lambda M_E \succeq 0 $. But then $ L+\lambda E \in C_f$, since clearly $ (L+\lambda E)(f) = 0 $ and the moment matrix of $ L+\lambda E $ is psd. Thus $ L = \left(L+\lambda E\right) - \lambda E  $ lies in the span of $ C_f $. 
\end{proof}

\begin{corollary}\label{cor:real-radical=>dual-nondegenerate} Let $ f\in \Sigma_{2k} $ and let $ I $ an ideal with $ f\in I \subseteq (\sosupp f) $. If $ I $ is real radical, then $ I  = (\sosupp f) $ and $ f $ has dual nondegenerate Sum-of-Squares support. 
\end{corollary}
\begin{proof}
	Since $ I $ is real radical, for all $ p_1,\ldots,p_N\in \R[X]_{k} $ such that $ p_1^2  + \ldots + p_N^2 \in I_{2k} $, it holds $ p_1,\ldots,p_m \in I $. This shows that in fact $ I = (\sosupp f) $. In addition, by \Cref{prop:dual-nondegeneracy}, it also shows that $ f $ has dual nondegenerate Sum-of-Squares support. 
\end{proof}

\section{Typical regions with singleton Gram spectrahedra}\label{app:typical-regions-gram-spec}

The condition of \Cref{cor:pof-recovery-interpretation-variety}(a) is satisfied on a Euclidean open subset. This can e.g. be seen from the Poincaré-Miranda theorem:

\begin{theorem}[{Poincaré-Miranda, cf. \cite[Introduction]{Fonda_Gidoni_2015}}]\label{thm:poincare-miranda}  
	Write $ \mathcal{H} $ for the parallelepiped spanned by linearly independent vectors $ v_1,\ldots, v_n\in \R^n $ and let $ f\colon \mathcal{H} \to \R^n,x\mapsto (f_1(x),\ldots, f_n(x))  $ a  continuous function. Denote 
	\begin{align*}
		\mathcal{H}_i^{1} := \{\sum_{j=1}^m \lambda_j v_j \mid \lambda_j\in [0, 1], \lambda_i = 1 \}, \quad \mathcal{H}_i^{0} := \{\sum_{j=1}^m \lambda_j v_j \mid \lambda_j\in [0, 1], \lambda_i = 0 \}
	\end{align*}
	for $ i\in \{1,\ldots, n\} $. Note these are the facets of $ \mathcal{H} $. Assume that for each $ i\in \{1,\ldots, n\} $, $ f_i \le 0 $ on $ \mathcal{H}_i^{0} $, but $ f_i \ge 0 $ on $ \mathcal{H}_i^{1} $. Then $ f $ has a zero on $ \mathcal{H} $. 
\end{theorem}

\noindent
We may now fill the gap that was left in the proof of \Cref{cor:pof-recovery-interpretation-variety}(a):

\begin{proof}[Proof of the addendum in \Cref{cor:pof-recovery-interpretation-variety}(a)]
	By \Cref{thm:poincare-miranda}, it clearly suffices to find, say, a rectangle $ \mathcal{H} $ and quadratics $ q_1,\ldots, q_m, q_{m+1},\ldots, q_n $ such that for each $ i\le m $, $ q_i < 0 $ on $ \mathcal{H}_i^{0} $ but $ q_i > 0 $ on $ \mathcal{H}_i^{1} $, as then the condition of \Cref{thm:poincare-miranda} will be satisfied in a neighbourhood $ \mathcal{U}_1\times \ldots \times \mathcal{U}_n \times \{q_{m+1}\} \times \ldots \times  \{q_n\} $ of $ (q_1,\ldots, q_m) $. By introducing some new variable $ Y $, one can take e.g. $ m = n $, $ q_{n+1} := 0 $ and $ q_i = (X_i + Y)^2-2Y^2 \in \R[X_1,\ldots, X_n, Y] $, for $ i\in \{1,\ldots, n\} $ where $ Y $ is another unknown, and consider the rectangle $ \mathcal{H} = [0, 1]^n \times \{1\} $ in the affine plane where ``$ Y = 1 $''. This corresponds to choosing the basis $ e_1 + e_{n+1},\ldots,e_n + e_{n+1}, e_{n+1}  $ in \Cref{thm:poincare-miranda}. Note $ \mathcal{H}_i^{a} = [0, 1]^{i-1} \times \{a\} \times [0, 1]^{n-i} \times  \{1\} $ for $ a\in \{0, 1\} $. We have $ q_i = -1 < 0 $ on $ \mathcal{H}_i^{0} $ and $ q_i = 2 > 0 $ on $ \mathcal{H}_i^{1} $. Thus, the condition $ V_{\R}(q_1,\ldots, q_m) \neq \emptyset $ is a typical property. 
\end{proof}